\documentclass[10pt,a4paper,twoside]{amsart}
\usepackage[utf8]{inputenc}

\newcommand{\PrR}{\Pr^{\mathrm{R}}}
\newcommand{\PrLk}{\Pr^{\mathrm{L}}_{\kappa}}
\newcommand{\PrRk}{\Pr^{\mathrm{R}}_{\kappa}}

\newcommand{\MapC}{\Map_{\mathcal{C}}}
\newcommand{\MapCc}{\Map_{\mathcal{C}}^{\txt{Cart}}}
\newcommand{\MAP}{\txt{MAP}}

\newcommand{\Gpd}{\txt{Gpd}}
\newcommand{\CAT}{\txt{CAT}}
\newcommand{\To}{\Rightarrow}

\newcommand{\CatIC}{\Cat_{\infty/\mathcal{C}}}
\newcommand{\CatICc}{\CatIC^{\txt{cart}}}
\newcommand{\CatIDc}{\Cat_{\infty/\mathcal{D}}^{\txt{cart}}}

\newcommand{\Tw}{\txt{Tw}}
\usepackage[sec,skipfonts,skipmbb,skipname]{enrabbrev}

\usepackage{eucal}

\usepackage{tikz-cd}

\def\firstpage{1}\def\lastpage{42}
\def\received{}\def\revised{}
\def\communicated{}

\makeatletter
\def\magnification{\afterassignment\m@g\count@}
\def\m@g{\mag=\count@\hsize6.5truein\vsize8.9truein\dimen\footins8truein}
\makeatother

\font\eightrm=cmr8
\font\caps=cmcsc10                    
\font\Caps=cmcsc10 scaled \magstep1   

\makeatletter
\@specialpagefalse\headheight=8.5pt
\makeatother

\def\TSkip{\bigskip}
\newbox\TheTitle{\obeylines\gdef\GetTitle #1
\ShortTitle  #2
\SubTitle    #3
\Author      #4
\ShortAuthor #5
\EndTitle
{\setbox\TheTitle=\vbox{\baselineskip=20pt\let\par=\cr\obeylines%
\halign{\centerline{\Caps##}\cr\noalign{\medskip}\cr#1\cr}}%
	\copy\TheTitle\TSkip\TSkip%
\def\next{#2}\ifx\next\empty\gdef\STitle{#1}\else\gdef\STitle{#2}\fi%
\def\next{#3}\ifx\next\empty%
    \else\setbox\TheTitle=\vbox{\baselineskip=20pt\let\par=\cr\obeylines%
    \halign{\centerline{\caps##} #3\cr}}\copy\TheTitle\TSkip\TSkip\fi%
\centerline{\caps #4}\TSkip\TSkip%
\def\next{#5}\ifx\next\empty\gdef\SAuthor{#4}\else\gdef\SAuthor{#5}\fi%
\ifx\received\empty\relax
    \else\centerline{\eightrm Received: \received}\fi%
\ifx\revised\empty\TSkip%
    \else\centerline{\eightrm Revised: \revised}\TSkip\fi%
\ifx\communicated\empty\relax
    \else\centerline{\eightrm Communicated by \communicated}\fi\TSkip\TSkip%
\catcode'015=5}}\def\Title{\obeylines\GetTitle}
\def\Abstract{\begingroup\narrower
    \parskip=\medskipamount\parindent=0pt{\caps Abstract. }}
\def\EndAbstract{\par\endgroup\TSkip}

\long\def\MSC#1\EndMSC{\def\arg{#1}\ifx\arg\empty\relax\else
     {\par\narrower\noindent%
     2010 Mathematics Subject Classification: #1\par}\fi}

\long\def\KEY#1\EndKEY{\def\arg{#1}\ifx\arg\empty\relax\else
	{\par\narrower\noindent Keywords and Phrases: #1\par}\fi\TSkip}

\newbox\TheAdd\def\Addresses{\vfill\copy\TheAdd\vfill
    \ifodd\number\lastpage\vfill\eject\phantom{.}\vfill\eject\fi}
{\obeylines\gdef\GetAddress #1
\Address #2 
\Address #3
\Address #4
\EndAddress
{\def\xs{5.9truecm}\parindent=0pt
\setbox0=\vtop{{\obeylines\hsize=\xs#1\par}}\def\next{#2}
\ifx\next\empty 
     \setbox\TheAdd=\hbox to\hsize{\hfill\copy0\hfill}
\else\setbox1=\vtop{{\obeylines\hsize=\xs#2\par}}\def\next{#3}
\ifx\next\empty 
     \setbox\TheAdd=\hbox to\hsize{\hfill\copy0\hfill\copy1\hfill}
\else\setbox2=\vtop{{\obeylines\hsize=\xs#3\par}}\def\next{#4}
\ifx\next\empty\ 
     \setbox\TheAdd=\vtop{\hbox to\hsize{\hfill\copy0\hfill\copy1\hfill}
                \vskip20pt\hbox to\hsize{\hfill\copy2\hfill}}
\else\setbox3=\vtop{{\obeylines\hsize=\xs#4\par}}
     \setbox\TheAdd=\vtop{\hbox to\hsize{\hfill\copy0\hfill\copy1\hfill}
	        \vskip20pt\hbox to\hsize{\hfill\copy2\hfill\copy3\hfill}}
\fi\fi\fi\catcode'015=5}}\gdef\Address{\obeylines\GetAddress}

\begin{document}
\def\LOCAL{\jobname.files}
\Title
Lax Colimits and Free Fibrations in $\infty$-Categories
\ShortTitle 
Lax Colimits and Free Fibrations in $\infty$-Categories
\SubTitle   
\Author 
David Gepner, Rune Haugseng, and Thomas Nikolaus
\ShortAuthor 
David Gepner, Rune Haugseng, Thomas Nikolaus
\EndTitle
\Abstract 
We define and discuss lax and weighted colimits of diagrams in 
$\infty$-categories and show that the coCartesian fibration corresponding to a 
functor is given by its lax colimit. A key ingredient, of independent interest, 
is a simple characterization of the free Cartesian fibration on a functor of 
$\infty$-categories. As an application of these results, we prove that 2-representable 
functors are preserved under exponentiation, and also that the total space of a 
presentable Cartesian fibration between  is presentable, generalizing a 
theorem of Makkai and Paré to the $\infty$-categories setting. Lastly, in an appendix, we 
observe that pseudofunctors between (2,1)-categories give rise to functors 
between $\infty$-categories via the Duskin nerve. 
setting and the Duskin nerve. 
\EndAbstract
\MSC 
    18D30, 18A30
\EndMSC
\KEY 
Fibered categories, presentable categories, lax limits and colimits
\EndKEY
\Address
David Gepner
Department of Mathematics
Purdue University
West Lafayette
Indiana
USA
dgepner@purdue.edu
\Address
Rune Haugseng
University of Copenhagen
Copenhagen
Denmark
haugseng@math.ku.dk
\Address
Thomas Nikolaus
Max-Planck-Institut für Mathematik
Bonn
Germany
\Address
\EndAddress


\tableofcontents

\section{Introduction}
In the context of ordinary category theory, Grothendieck's theory of
fibrations~\cite{SGA1} can be used to give an alternative description
of functors to the category $\Cat$ of categories. This has been
useful, for example, in the theory of stacks in algebraic geometry, as
the fibration setup is usually more flexible. When working with
\icats{}, however, the analogous notion of \emph{Cartesian} fibrations
is far more important: since defining a functor to the \icat{} $\CatI$
of \icats{} requires specifying an infinite amount of coherence data,
it is in general not feasible to ``write down'' definitions of
functors, so that manipulating Cartesian fibrations is often the only
reasonable way to define key functors.

For ordinary categories, the \emph{Grothendieck construction} gives a
simple description of the fibration classified by a functor $F \colon
\mathbf{C}^{\op} \to \Cat$; this can also be described formally as a
certain weighted colimit, namely the lax colimit of the functor
$F$. For \icats{}, on the other hand, the equivalence between
Cartesian fibrations and functors has been proved by Lurie using the
\emph{straightening functor}, a certain left Quillen functor between
model categories. This leaves the corresponding right adjoint, the
\emph{unstraightening} functor, quite inexplicit. 

One of our main goals in this paper is to show that Lurie's
unstraightening functor is a model for the \icatl{} analogue of the
Grothendieck construction. More precisely, we introduce \icatl{}
versions of lax and oplax limits and colimits and prove the following:
\begin{thm}\ 
  \begin{enumerate}[(i)]
  \item Suppose $F \colon \mathcal{C} \to \CatI$ is a functor of
    \icats{}, and $\mathcal{E} \to \mathcal{C}$ is a coCartesian
    fibration classified by $F$. Then $\mathcal{E}$ is the oplax
    colimit of the functor $F$.
  \item Suppose $F \colon \mathcal{C}^{\op} \to \CatI$ is a functor of
    \icats{}, and $\mathcal{E} \to \mathcal{C}$ is a Cartesian
    fibration classified by $F$. Then $\mathcal{E}$ is the lax
    colimit of the functor $F$.
  \end{enumerate}
\end{thm}

To prove this we make use of an explicit description of the free
Cartesian fibration on an arbitary functor of \icats{}.  More
precisely, the \icat{} $\CatICc$ of Cartesian fibrations over
$\mathcal{C}$ is a subcategory of the slice \icat{} $\CatIC$, and we
show that the inclusion admits a left adjoint given by a simple
formula:
\begin{thm}
  Let $\mathcal{C}$ be an \icat{}. For $p \colon \mathcal{E} \to
  \mathcal{C}$ any functor of \icats{}, let $F(p)$ denote the map
  $\mathcal{E} \times_{\mathcal{C}^{\{1\}}} \mathcal{C}^{\Delta^{1}} \to
  \mathcal{C}^{\{0\}}$ (i.e.\ the pullback is along the map
  $\mathcal{C}^{\Delta^{1}} \to \mathcal{C}$ given by evaluation at $1
  \in \Delta^{1}$ and the projection is induced by evaluation at
  $0$). Then $F$ defines a functor $\CatIC \to \CatICc$, which is left
  adjoint to the forgetful functor $\CatICc \to \CatIC$.
\end{thm}

In the special case where $p\colon\mathcal{E}\to\mathcal{C}$ is a Cartesian
fibration and $\mathcal{C}$ is an $\infty$-category equipped with a
``mapping $\infty$-category'' functor
$\MAP_\mathcal{C}\colon\mathcal{C}^{\op}\times\mathcal{C}\to\CatI$,
such as is the case when $\mathcal{C}$ is the underlying \icat{} of an
$(\infty,2)$-category, then it is natural to ask when $p$ is classified
by a functor of the form
$\MAP_\mathcal{C}(-,X)\colon\mathcal{C}^{\op}\to\CatI$ for some object
$X$ of $\mathcal{C}$.  We say that $p$ is {\em 2-representable} when
this is the case.  As an application of our theorems, we show that, if
$p\colon\mathcal{E}\to\mathcal{C}$ is a 2-representable Cartesian
fibration such that the mapping \icat{} functor $\MAP_\mathcal{C}$ is
tensored and cotensored over $\CatI$, and $\mathcal{D}$ is any
\icat{}, then the exponential
$q\colon\Fun(\mathcal{D},\mathcal{E})\to \Fun(\mathcal{D},\mathcal{C})$ is itself
2-representable.  This is relevant in the description of the functoriality of twisted cohomology theories as discussed in joint work of the third author with U. Bunke \cite{Differentialcohomology}. More precisely it describes a converse to the construction in Section 3 and Appendix A of this paper.

The third main result of this paper provides a useful extension of the
theory of presentable \icats{} in the context of Cartesian fibrations,
generalizing a theorem of Makkai and Paré~\cite{MakkaiPare} to the
\icatl{} context. More precisely, we show:
\begin{thm}
  Suppose $p \colon \mathcal{E} \to \mathcal{C}$ is a Cartesian and
  coCartesian fibration such that $\mathcal{C}$ is presentable, the
  fibres $\mathcal{E}_{x}$ are presentable for all $x \in
  \mathcal{C}$, and the classifying functor $F \colon \mathcal{C}^{\op}
  \to \LCatI$ preserves $\kappa$-filtered limits for some regular cardinal
  $\kappa$. Then the \icat{} $\mathcal{E}$ is presentable, and the
  projection $p$ is an accessible functor (i.e. it preserves
  $\lambda$-filtered colimits for some sufficiently large cardinal
  $\lambda$).
\end{thm}
While the theory of \emph{accessible} and \emph{presentable}
categories is already an important part of ordinary category theory,
when working with \icats{} the analogous notions turn out to be
indispensable. Whereas, for example, it is often possible to give an
explicit construction of colimits in an ordinary category, when
working with \icats{} we often have to conclude that colimits exist by
applying general results on presentable \icats{}. Similarly, while for
ordinary categories one can frequently just write down an adjoint to a
given functor, for \icats{} an appeal to the adjoint functor theorem,
which is most naturally considered in the presentable context, is
often unavoidable. It is thus very useful to know that various ways of
constructing \icats{} give accessible or presentable ones; many such
results are proved in \cite{HTT}*{\S 5}, and our result adds to these
by giving a criterion for the source of a Cartesian fibration to be
presentable.

\subsection{Overview}
In \S\ref{sec:twist} we briefly review the definitions of twisted
arrow \icats{} and \icatl{} ends and coends, and use these to define
weighted (co)limits. Then in \S\ref{sec:simplex} we prove our main
result for coCartesian fibrations over a simplex, using the
\emph{mapping simplex} defined in \cite{HTT}*{\S 3.2.2}. Before we
extend this result to general coCartesian fibrations we devote three
sections to preliminary results: in \S\ref{sec:freefib} we give a
description of the \emph{free Cartesian fibration}, i.e. the left
adjoint to the forgetful functor from Cartesian fibrations over
$\mathcal{C}$ to the slice \icat{} $\CatIC$; in \S\ref{sec:nattrend}
we prove that the space of natural transformations between two
functors is given by an end (a result first proved by Glasman \cite[Proposition 2.3.]{GlasmanTHHHodge}), and in \S\ref{sec:enhmap} we prove that
the straightening equivalence extends to an equivalence of the natural
enrichments in $\CatI$ of the two \icats{}
involved. \S\ref{sec:Cartwtcolim} then contains the proof of our main
result: Cartesian and coCartesian fibrations are given by weighted
colimits of the classifying functors. In \S\ref{sec:laxrep} we give a
simple application of our results to functors that are representable
via an enrichment in $\CatI$, and in \S\ref{sec:proof} we apply them
to identify the functor classifying a certain simple Cartesian
fibration; this is a key step in our proof in \S\ref{sec:pres} that
the source of a presentable fibration is presentable. Finally, in
appendix~\ref{sec:pseudo} we use Duskin's nerve for strict
(2,1)-categories to check that the pseudonaturality of the
unstraightening functors on the level of model categories implies that
they are natural on the level of \icats{}.

\subsection{Notation}
Much of this paper is based on work of Lurie in \cite{HTT,HA}; we have
generally kept his notation and terminology. In particular, by an
\emph{\icat{}} we mean an $(\infty,1)$-category or more specifically a
quasicategory.
 We also use the following conventions, some of which differ from those
of Lurie:
\begin{itemize}
\item Generic categories are generally denoted by single capital
  bold-face letters ($\mathbf{A},\mathbf{B},\mathbf{C}$) and generic
  \icats{} by single caligraphic letters
  ($\mathcal{A},\mathcal{B},\mathcal{C}$). Specific categories and
  \icats{} both get names in the normal text font.
\item If $\mathcal{C}$ is an \icat{}, we write $\iota \mathcal{C}$ for
  the \emph{interior} or \emph{underlying space} of $\mathcal{C}$,
  i.e.\ the largest subspace of $\mathcal{C}$ that is a Kan complex.
\item If $f \colon \mathcal{C} \to \mathcal{D}$ is left adjoint to a
  functor $g \colon \mathcal{D} \to \mathcal{C}$, we will refer to the
  adjunction as $f \dashv g$.
\item We write $\PrL$ for the \icat{} of presentable \icats{} and
  functors that are left adjoints, i.e.\ colimit-preserving functors,
  and $\PrR$ for the \icat{} of presentable \icats{} and
  functors that are right adjoints, i.e.\ accessible functors that
  preserve all small limits.
\item If $\mathcal{C}$ and $\mathcal{D}$ are \icats{}, we will denote
  the \icat{} of functors $\mathcal{C} \to \mathcal{D}$ by both
  $\Fun(\mathcal{C}, \mathcal{D})$ and $\mathcal{D}^{\mathcal{C}}$.
\item If $S$ is a simplicial set, we write 
  \[ \txt{St}^{+}_{S} : (\sSet^{+})_{/S^{\sharp}}  \rightleftarrows
  \Fun(\mathfrak{C}(S)^{\op}, \sSet^{+}) : \txt{Un}^{+}_{S} \]
  for the marked (un)straightening Quillen equivalence, as
  defined in \cite{HTT}*{\S 3.2}.
  \item We write $\CatICc$ for the subcategory of $\CatIC$ consisting of
  Cartesian fibrations over $\mathcal{C}$, with morphisms the functors
  that preserve Cartesian edges, $\MapCc(\blank,\blank)$ for the
  mapping spaces in $\CatICc$, and
  $\Fun_{\mathcal{C}}^{\txt{cart}}(\blank, \blank)$ for the \icat{} of
  functors that preserve Cartesian edges, defined as a full
  subcategory of the \icat{} $\Fun_{\mathcal{C}}(\blank, \blank)$ of
  functors over $\mathcal{C}$. Similarly, we write
  $\CatIC^{\txt{cocart}}$ for the \icat{} of coCartesian fibrations
  over $\mathcal{C}$, $\Map_{\mathcal{C}}^{\txt{cocart}}(\blank,
  \blank)$ for the mapping spaces in $\CatIC^{\txt{cocart}}$, and
  $\Fun_{\mathcal{C}}^{\txt{cocart}}(\blank, \blank)$ for the full
  subcategory of $\Fun_{\mathcal{C}}(\blank, \blank)$ spanned by the
  functors that preserve coCartesian edges.

\item If $\mathcal{C}$ is an \icat{}, we write 
  \[ \txt{St}_{\mathcal{C}} \colon \CatICc \rightleftarrows
  \Fun(\mathcal{C}^{\op}, \CatI) \colon \txt{Un}_{\mathcal{C}} \]
  for the adjoint equivalence of \icats{} induced by the
  (un)straightening Quillen equivalence via \cite{HTT}*{Proposition 5.2.4.6}.
\item If $S$ is a simplicial set, we write 
  \[ \txt{St}^{+,\txt{co}}_{S} : (\sSet^{+})_{/S^{\sharp}}
  \rightleftarrows \Fun(\mathfrak{C}(S), \sSet^{+}) :
  \txt{Un}^{+,\txt{co}}_{S} \] for the coCartesian marked
  (un)straightening Quillen equivalence, given by
  $\txt{St}^{+,\txt{co}}_{S}(X) :=
  (\txt{St}^{+}_{S^{\op}}(X^{\op}))^{\op}$.
\item If $\mathcal{C}$ is an \icat{}, we write 
  \[ \txt{St}^{\txt{co}}_{\mathcal{C}} \colon \CatIC^{\txt{cocart}} \rightleftarrows
  \Fun(\mathcal{C}, \CatI) \colon \txt{Un}^{\txt{co}}_{\mathcal{C}} \]
  for the adjoint equivalence of \icats{} induced by the
  coCartesian (un)straightening Quillen equivalence.
\item If $\mathcal{C}$ is an \icat{}, we denote the Yoneda embedding
  for $\mathcal{C}$
  by \[ y_{\mathcal{C}} \colon \mathcal{C} \to
  \mathcal{P}(\mathcal{C}),\]
  where $\mathcal{P}(\mathcal{C})$ is the presheaf \icat{}
  $\Fun(\mathcal{C}^{\op}, \mathcal{S})$ with $\mathcal{S}$ the
  \icat{} of spaces.
\end{itemize}

\subsection{Acknowledgments}
\emph{David}: Thanks to Joachim Kock for helpful discussions regarding free fibrations and lax colimits.\\
\emph{Rune}: I thank Clark Barwick for helpful discussions of the
presentability result and Michael Shulman for telling me about
\cite{MakkaiPare}*{Theorem 5.3.4} in answer to a MathOverflow
question.

We thank Aaron Mazel-Gee and Omar Antolín Camarena for pointing out
some inaccuracies in the first version of this paper. We also thank an anonymous referee for a very careful and helpful report.

\section{Twisted Arrow $\infty$-Categories, (Co)ends, and Weighted (Co)limits}\label{sec:twist}
In this section we briefly recall the definitions of twisted arrow
\icats{} and (co)ends, and then use these to give a natural
definition of weighted (co)limits in the \icatl{} setting.

\begin{defn}
  Let $\epsilon \colon \simp \to \simp$ be the functor
  $[n] \mapsto [n] \star [n]^{\op}$. The \emph{edgewise subdivision}
  of a simplicial set $S$ is the composite
  $\epsilon^{*}S = S \circ \epsilon$.
\end{defn}

\begin{defn}
  Let $\mathcal{C}$ be an \icat{}. The \emph{twisted arrow \icat{}}
  $\txt{Tw}(\mathcal{C})$ of $\mathcal{C}$ is the simplicial set
  $\epsilon^{*}\mathcal{C}$. Thus in particular
  \[ \Hom(\Delta^{n}, \txt{Tw}(\mathcal{C})) \cong \Hom(\Delta^{n}
  \star (\Delta^{n})^{\op}, \mathcal{C}).\] The natural transformations
  $\Delta^{\bullet}, (\Delta^{\bullet})^{\op} \to \Delta^{\bullet}
  \star (\Delta^{\bullet})^{\op}$ induce a projection
  $\txt{Tw}(\mathcal{C}) \to \mathcal{C} \times \mathcal{C}^{\op}$.
\end{defn}

\begin{remark}
  The twisted arrow \icat{}, which was originally introduced by Joyal,
  has previously been extensively used by
  Barwick~\cite{BarwickQ,BarwickMackey} and
  collaborators~\cite{BarwickGlasmanNardinCart}, and by 
  Lurie~\cite{HA}*{\S 5.2.1}. By \cite{HA}*{Proposition 5.2.1.3} the
  projection $\txt{Tw}(\mathcal{C}) \to \mathcal{C} \times
  \mathcal{C}^{\op}$ is a right fibration; in particular, the
  simplicial set $\txt{Tw}(\mathcal{C})$ is an \icat{} if
  $\mathcal{C}$ is. The functor $\mathcal{C}^{\op} \times \mathcal{C}
  \to \mathcal{S}$ classified by this right fibration is the mapping
  space functor $\Map_{\mathcal{C}}(\blank, \blank)$ by
  \cite{HA}*{Proposition 5.2.1.11}.
\end{remark}

\begin{warning}
  There are two possible conventions for defining the edgewise
  subdivision (and therefore also the twisted arrow \icat{}); we
  follow that of Lurie in \cite{HA}*{\S 5.2.1}. Alternatively, one can
  define the edgewise subdivision using the functor $[n] \mapsto
  [n]^{\op} \star [n]$, in which case $\Tw(\mathcal{C}) \to
  \mathcal{C}^{\op} \times \mathcal{C}$ is a \emph{left} fibration ---
  this is the convention used in the papers of Barwick cited above.
\end{warning}

\begin{ex}\label{ex:Twn}
  The twisted arrow category $\txt{Tw}([n])$ of the category $[n]$ is
  the partially ordered set with objects $(i,j)$ where $0 \leq i \leq
  j \leq n$ and with $(i,j) \leq (i',j')$ if $i \leq i' \leq j' \leq
  j$.
\end{ex}

A natural definition of (co)ends in the \icatl{} setting is then the
following.
\begin{defn}
  If $F \colon \mathcal{C} \times \mathcal{C}^{\op} \to \mathcal{D}$
  is a functor of \icats{}, the \emph{coend} of $F$
  is the colimit of the
  composite functor
  \[ \txt{Tw}(\mathcal{C}) \to \mathcal{C} \times \mathcal{C}^{\op}
  \to \mathcal{D}. \]
  Similarly, if $G\colon \mathcal{C}^\op \times \mathcal{C}\to \mathcal{D}$ is a functor of \icats{}, then the \emph{end} of $G$ is the limit of the composite functor
  \[ \txt{Tw}(\mathcal{C})^{\op} \to \mathcal{C}^\op \times \mathcal{C}
  \to \mathcal{D}.\]
\end{defn}

\begin{remark}
  These \icatl{} notions of ends and coends are also discussed in
  \cite{GlasmanTHHHodge}*{\S 2}. In the context of simplicial
  categories, a homotopically correct notion of coends was extensively
  used by Cordier and Porter~\cite{CordierPorter}; see their paper for
  a discussion of the history of such definitions.
\end{remark}

Now we can consider weighted (co)limits:
\begin{defn}
  Let $\mathcal{R}$ be a presentably symmetric monoidal \icat{},
  i.e.\ a presentable \icat{} equipped with a symmetric monoidal
  structure such that the tensor product preserves colimits in each
  variable, and let $\mathcal{M}$ be a right $\mathcal{R}$-module
  in $\PrL$. Then $\mathcal{M}$ is in particular tensored and
  cotensored over $\mathcal{R}$, i.e.\ there are functors
  \[ (\blank \otimes \blank) \colon \mathcal{M}\times \mathcal{R} \to \mathcal{M},\]
  \[ (\blank)^{(\blank)} \colon \mathcal{R}^{\op} \times \mathcal{M}
  \to \mathcal{M},\] such that for every $x \in \mathcal{R}$ the
  functor $\blank \otimes x \colon \mathcal{M} \to \mathcal{M}$ is
  left adjoint to $(\blank)^{x}$. Given functors $F \colon \mathcal{C}
  \to \mathcal{M}$ and $W \colon \mathcal{C}^{\op} \to \mathcal{R}$,
  the \emph{$W$-weighted colimit} $\colim^{W}_{\mathcal{C}} F$ of $F$ is defined to be the coend
  $\colim_{\txt{Tw}(\mathcal{C})} F(\blank) \otimes W(\blank)$. Similarly, given $F \colon \mathcal{C} \to \mathcal{M}$
  and $W \colon \mathcal{C} \to \mathcal{R}$, the \emph{$W$-weighted
    limit} $\lim^{W}_{\mathcal{C}} F$ of $F$ is the end $\lim_{\txt{Tw}(\mathcal{C})^{\op}}
  F(\blank)^{W(\blank)}$.
\end{defn}

We are interested in the case where both $\mathcal{R}$ and
$\mathcal{M}$ are the \icat{} $\CatI$ of \icats{}, with the tensoring
given by Cartesian product and the cotensoring by $\Fun(\blank,
\blank)$. In this case there are two special weights for every \icat{}
$\mathcal{C}$: we have functors $\mathcal{C}_{/\blank} \colon
\mathcal{C} \to \CatI$ and $\mathcal{C}_{\blank/} \colon
\mathcal{C}^{\op} \to \CatI$ sending $x \in \mathcal{C}$ to
$\mathcal{C}_{/x}$ and $\mathcal{C}_{x/}$, respectively. Precisely,
these functors are obtained by straightening the source and target
projections $\mathcal{C}^{\Delta^{1}} \to \mathcal{C}$, which are
respectively Cartesian and coCartesian. Using these functors, we can
define lax and oplax (co)limits:
\begin{defn}
  Suppose $F \colon \mathcal{C} \to \CatI$ is a
  functor. Then:
  \begin{itemize}
  \item The \emph{oplax colimit} of $F$ is the colimit of $F$ weighted
    by $\mathcal{C}_{\blank/}$, i.e. \[\colim_{\txt{Tw}(\mathcal{C})}
    F(\blank) \times \mathcal{C}_{\blank/}.\]
  \item The \emph{lax colimit} of $F$ is the colimit of $F$ weighted
    by $(\mathcal{C}^{\op})_{/\blank}$, i.e. \[\colim_{\txt{Tw}(\mathcal{C})}
    F(\blank) \times (\mathcal{C}^\op)_{/\blank} .\]
  \item The \emph{lax limit} of $F$ is the limit of $F$ weighted by
    $\mathcal{C}_{/\blank}$, i.e. \[\lim_{\txt{Tw}(\mathcal{C})^\op}
    \Fun(\mathcal{C}_{/\blank}, F(\blank)).\]
  \item The \emph{oplax limit} of $F$ is the limit of $F$ weighted by
    $(\mathcal{C}^{\op})_{\blank/}$, i.e. \[\lim_{\txt{Tw}(\mathcal{C})^\op}
    \Fun((\mathcal{C}^\op)_{\blank/}, F(\blank)).\]
  \end{itemize}
\end{defn}

\section{CoCartesian Fibrations over a Simplex}\label{sec:simplex}
In this preliminary section we study coCartesian fibrations over the
simplices $\Delta^{n}$, and observe that in this case the description
of a coCartesian fibration as an oplax colimit follows easily from
results of Lurie in \cite{HTT}*{\S 3.2}. More precisely, we will
prove:
\begin{propn}\label{propn:sxweighted}
  There is an equivalence
  \[ \colim_{\txt{Tw}([n])} \phi(\blank) \times [n]_{\blank/}    \isoto
  \txt{Un}^{\txt{co}}_{[n]}(\phi)\] of functors $\Fun([n], \CatI) \to \CatI$,
  natural in $\simp^{\op}$.  
\end{propn}
To see this we first recall fron \cite[\S 3.2]{HTT}
the definition and some features of the \emph{mapping simplex} of a functor $\phi \colon
[n] \to \sSet^{+}$ and show that its fibrant replacement is a
coCartesian fibration classified the corresponding functor
$\Delta^{n} \to \CatI$.

\begin{defn}
  Let $\phi \colon [n] \to \sSet^{+}$ be a functor. The \emph{mapping
    simplex} $M_{[n]}(\phi) \to \Delta^{n}$ has $k$-simplices given by
  a map $\sigma \colon [k] \to [n]$ together with a $k$-simplex
  $\Delta^{k} \to \phi(\sigma(0))$. In particular, an edge of
  $M_{[n]}(\phi)$ is given by a pair of integers $0 \leq i \leq j \leq
  n$ and an edge $f$ in $\phi(i)$; let $S$ be the set of edges of
  $M_{[n]}(\phi)$ where the edge $f$ is marked. Then
  $M^{\natural}_{[n]}(\phi)$ is the marked simplicial set
  $(M_{[n]}(\phi), S)$. This gives a functor $M^{\natural}_{[n]}
  \colon \Fun([n], \sSet^{+}) \to (\sSet^{+})_{/\Delta^{n}}$,
  pseudonatural in $\simp^{\op}$ (with respect to composition and
  pullback) --- see Appendix~\ref{sec:pseudo} for a discussion of
  pseudonatural transformations.
\end{defn}

\begin{defn}
  Let $\phi \colon [n] \to \sSet$ be a functor. The \emph{relative
    nerve} $N_{[n]}(\phi) \to \Delta^{n}$ has $k$-simplices given by
  a map $\sigma \colon [k] \to [n]$ and for every ordered subset $J
  \subseteq [k]$ with greatest element $j$, a map $\Delta^{J} \to
  \phi(\sigma(j))$ such that for $J' \subseteq J$ the diagram
  \nolabelcsquare{\Delta^{J'}}{\phi(\sigma(j'))}{\Delta^{J}}{\phi(\sigma(j))}
  commutes. Given a functor $\overline{\phi} \colon
  [n] \to \sSet^{+}$ we define
  $\mathrm{N}_{[n]}^{+}(\overline{\phi})$ to be the marked simplicial set
  $(\mathrm{N}_{[n]}(\phi), M)$ where $\phi$ is the underlying functor
  $[n] \to \sSet$ of $\overline{\phi}$, and $M$ is the set of edges
  $\Delta^{1} \to \mathrm{N}_{[n]}\phi$ determined by
  \begin{itemize}
  \item a pair of integers $0 \leq i \leq j \leq n$,
  \item a vertex $x \in \phi(i)$,
  \item a vertex $y \in \phi(j)$ and an edge $\phi(i \to j)(x) \to y$ that
    is marked in $\overline{\phi}(j)$.
  \end{itemize}
  This determines a functor $\mathrm{N}_{[n]}^{+}\colon \Fun([n],
  \sSet^{+}) \to (\sSet^{+})_{/\Delta^{n}}$, pseudonatural in
  $\simp^{\op}$.
\end{defn}

\begin{remark}
  By \cite{HTT}*{Proposition 3.2.5.18}, the functor $N^{+}_{[n]}$ is a
  right Quillen equivalence from the projective model structure on
  $\Fun([n], \sSet^{+})$ to the coCartesian model structure on
  $(\sSet^{+})_{/\Delta^{n}}$. In particular, if $\phi \colon [n] \to
  \sSet^{+}$ is a functor such that $\phi(i)$ is fibrant (i.e.\ is a
  quasicategory marked by its set of equivalences) for every $i$, then
  $N_{[n]}^{+}(\phi)$ is a coCartesian fibration.
\end{remark}

\begin{defn}
  There is a natural transformation $\nu_{[n]} \colon
  M_{[n]}^{\natural}(\blank) \to N_{[n]}^{+}(\blank)$ that sends a
  $k$-simplex $(\sigma \colon [k] \to [n], \Delta^{k} \to
  \phi(\sigma(0)))$ in $M_{[n]}^{\natural}(\sigma)$ to the $k$-simplex
  of $N^{+}_{[n]}(\phi)$ determined by the composites $\Delta^{J} \to
  \Delta^{k}\to \phi(\sigma(0)) \to \phi(\sigma(j))$. This is clearly
  pseudonatural in maps in $\simp^{\op}$, i.e.\ we have a pseudofunctor
  $\simp^{\op} \to \Fun([1], \Cat)$ that to $[n]$ assigns
  \[\nu_{[n]} \colon [1] \times \Fun([n], \sSet^{+}) \to
  (\sSet^{+})_{/\Delta^{n}}.\]
\end{defn}

\begin{propn}\label{propn:MNeq}
  Suppose $\phi \colon [n] \to \sSet^{+}$ is fibrant. Then the natural
  map $\nu_{[n],\phi} \colon M^{\natural}_{[n]}(\phi) \to N^{+}_{[n]}(\phi)$ is a
  coCartesian equivalence.
\end{propn}

\begin{proof}
  Since $N_{[n]}^{+}(\phi) \to \Delta^{n}$ is a coCartesian fibration
  by \cite{HTT}*{Proposition 3.2.5.18}, it follows from
  \cite{HTT}*{Proposition 3.2.2.14} that it suffices to check that
  $\nu_{[n],\phi}$ is a ``quasi-equivalence'' in the sense of
  \cite{HTT}*{Definition 3.2.2.6}. Thus we need only show that the
  induced map on fibres $M^{\natural}_{[n]}(\phi)_{i} \to N^{+}_{[n]}(\phi)_{i}$ is a
  categorical equivalence for all $i = 0,\ldots,n$. But unwinding the
  definitions we see that this can be identified with the identity map
  $\phi(i) \to \phi(i)$ (marked by the equivalences).
\end{proof}

Let $\txt{Un}^{+,\txt{co}}_{[n]} \colon \Fun([n], \sSet^{+}) \to
(\sSet^{+})_{/\Delta^{n}}$ be the coCartesian version of the marked
\emph{unstraightening functor} defined in \cite{HTT}*{\S 3.2.1}. By
\cite{HTT}*{Remark 3.2.5.16} there is a natural transformation
$\lambda_{[n]} \colon N^{+}_{[n]} \to \txt{Un}^{+,\txt{co}}_{[n]}$,
which is a weak equivalence on fibrant objects by
\cite{HTT}*{Corollary 3.2.5.20}. Since this is also pseudonatural in
$\simp^{\op}$, combining this with
Proposition~\ref{propn:MNeq} we immediately get:
\begin{cor}\label{cor:MUneq}
  For every $[n] \in \simp^{\op}$ there is a natural transformation
  $\lambda_{[n]}\nu_{[n]} \colon M_{[n]}^{\natural}(\blank) \to
  \txt{Un}^{+,\txt{co}}_{[n]}(\blank)$, and this is pseudonatural in
  $[n] \in \simp^{\op}$. If $\phi \colon [n] \to \sSet^{+}$ is
  fibrant, then the map $M_{[n]}^{\natural}(\phi) \to
  \txt{Un}^{+,\txt{co}}_{[n]}(\phi)$ is a coCartesian equivalence.
\end{cor}

It is immediate from the definition that $M_{[n]}^{\natural}(\phi)$ is the
pushout
\[
\begin{tikzcd}
  \phi(0)^{\natural} \times (\Delta^{\{1,\ldots,n\}})^{\sharp} \arrow[hookrightarrow]{r} \arrow{d} & \phi(0)^{\natural} \times
  (\Delta^{n})^{\sharp} \arrow{d}\\
  M^{\natural}_{[n-1]}(\phi|_{\{1,\ldots,n\}}) \arrow[hookrightarrow]{r} & M_{[n]}^{\natural}(\phi).
\end{tikzcd}
\]

Moreover, since all objects are cofibrant in the model structure on
marked simplicial sets and the top horizontal map is a cofibration,
this is a homotopy pushout. Combining this with
Corollary~\ref{cor:MUneq}, we get the following:
\begin{lemma}\label{lem:Dncocartpo} 
  Suppose $F \colon [n] \to \CatI$ is a functor, and that $\mathcal{E}
  \to \Delta^{n}$ is the corresponding coCartesian fibration. Let
  $\mathcal{E}'$ be the pullback of $\mathcal{E}$ along the inclusion
  $\Delta^{\{1,\ldots,n\}} \hookrightarrow \Delta^{n}$. Then there is
  a pushout square \nolabelcsquare{F(0) \times \Delta^{\{1,\ldots,n\}}}{F(0)
    \times \Delta^{n}}{\mathcal{E}'}{\mathcal{E}}
  in $\CatI$.
\end{lemma}

Unwinding the definition, we see that $M_{[n]}^{\natural}(\phi)$ is the colimit
of the diagram
\[
\begin{tikzcd}[column sep=tiny]
  & \phi(0) \times (\Delta^{n})^{\sharp} \\
\phi(0) \times (\Delta^{\{1,\ldots,n\}})^{\sharp} \arrow{ur} \arrow{dr}&  \\
 & \phi(1) \times (\Delta^{\{1,\ldots,n\}})^{\sharp} \\
\phi(1) \times (\Delta^{\{2,\ldots,n\}})^{\sharp}  \arrow{ur} \arrow{dr}& \\
 & \phi(2) \times (\Delta^{\{2,\ldots,n\}})^{\sharp} \\
 & \vdots \\
  & \phi(n-1) \times (\Delta^{\{n-1,n\}})^{\sharp} \\
\phi(n-1) \times (\Delta^{\{n\}})^{\sharp}  \arrow{ur} \arrow{dr}& \\
 & \phi(n) \times (\Delta^{\{n\}})^{\sharp}.
\end{tikzcd}
\]
By Example~\ref{ex:Twn} the category indexing this colimit is a
cofinal subcategory of the twisted arrow category $\txt{Tw}([n])$ of
$[n]$ --- this is easy to check using \cite{HTT}*{Corollary 4.1.3.3}
since both categories are partially ordered sets. Hence we may
identify $M^{\natural}_{[n]}(\phi)$ with the
coend
\[\colim_{\txt{Tw}([n])}(\phi(\blank) \times \mathrm{N}[n]_{\blank/})
.\]
Moreover, since we can write this colimit as an iterated pushout along
cofibrations, this is a homotopy colimit.  From this we can prove
Proposition~\ref{propn:sxweighted} using the results of
appendix~\ref{sec:pseudo} together with the following observation:
\begin{lemma}\label{lem:fibrant}
  Let $G \colon \mathbf{C} \to \mathbf{D}$ be a right Quillen functor
  between model categories. Suppose $f \colon X \to \bar{X}$ and $g \colon
  Y \to \bar{Y}$ are weak equivalences such that $\bar{X}$ and
  $\bar{Y}$ are fibrant, and $G(f)$ and $G(g)$ are weak equivalences
  in $\mathbf{D}$. Then if $h \colon X \to Y$ is a weak equivalence,
  the morphism $G(h)$ is also a weak equivalence in $\mathbf{D}$.
\end{lemma}
\begin{proof}
  Choose a factorization of the composite $g \circ h \colon X \to
  \bar{Y}$ as a trivial cofibration $i \colon X \hookrightarrow X'$
  followed by a fibration $p \colon X' \to \bar{Y}$. We then have a
  commutative diagram
  \[
  \begin{tikzcd}
    X \arrow{dd}{h} \arrow{rr}{f} \arrow{dr}{i} & & \bar{X} \\
    {} & X' \arrow{dr}{p} \arrow[dotted]{ur}{q} & \\
    Y \arrow{rr}{g} & & \bar{Y}
  \end{tikzcd}
  \]
  where the dotted arrow $q$ exists since $\bar{X}$ is fibrant and $i$ is
  a trivial cofibration, and all morphisms are weak equivalences by
  the 2-of-3 property. By assumption $G$ takes $f$ and $g$ to weak
  equivalences, and as $G$ is a right Quillen functor by Brown's Lemma it also takes $q$
  and $p$ to weak equivalences as these are weak equivalences between
  fibrant objects. By the 2-of-3 property we can then conclude first that
  $G(i)$ is a weak equivalence and then that $G(h)$ is a weak equivalence.
\end{proof}

\begin{proof}[Proof of Proposition~\ref{propn:sxweighted}]
  We will prove this by applying Proposition~\ref{propn:relGrothfib}
  to a relative Grothendieck fibration constructed in the same way as
  in Proposition~\ref{propn:unstrelGr}. The only difference is that
  the mapping simplex of a functor $\phi \colon [n] \to \sSet^{+}$ is
  not in general fibrant. We must therefore consider a larger relative
  subcategory of $(\sSet^{+})_{/\Delta^{n}}$ containing the mapping
  simplices of fibrant functors whose associated \icat{} is still
  $\Cat_{\infty/\Delta^{n}}^{\txt{cocart}}$.

  By \cite{HTT}*{Proposition 3.2.2.7} every mapping simplex admits a
  weak equivalence to a fibrant object that is preserved under
  pullbacks along all morphisms in $\simp$. We therefore think of
  $M_{[n]}^{\natural}$ and $\txt{Un}^{+,\txt{co}}_{[n]}$ as functors
  from fibrant objects in $\Fun([n], \sSet^{+})$ to objects in
  $(\sSet^{+})_{/\Delta^{n}}$ that admit a weak equivalence to a
  fibrant object that is preserved by pullbacks --- by
  Lemma~\ref{lem:fibrant} all weak equivalences between such objects
  are preserved by pullbacks, so we still get functors of relative
  categories. 

  It remains to show that inverting the weak equivalences in this
  subcategory gives the same \icat{} as inverting the weak
  equivalences in the subcategory of fibrant objects. This follows
  from \cite{BarwickKanRelCat}*{7.5}, since any fibrant replacement
  functor gives a homotopy equivalence of relative categories.
\end{proof}

\section{Free Fibrations}\label{sec:freefib}
Our goal in this section is to prove that for any \icat{} $\mathcal{C}$, the
forgetful functor \[\CatICc \to \CatIC\] has a left adjoint,
given by the following explicit formula:

\begin{defn}
  Let $\mathcal{C}$ be an \icat{}. For $p \colon \mathcal{E} \to
  \mathcal{C}$ any functor of \icats{}, let $F(p)$ denote the map
  $\mathcal{E} \times_{\mathcal{C}^{\{1\}}} \mathcal{C}^{\Delta^{1}} \to
  \mathcal{C}^{\{0\}}$, where the pullback is along the target fibration
  $\mathcal{C}^{\Delta^{1}} \to \mathcal{C}$ given by evaluation at $1
  \in \Delta^{1}$, and the projection $F(p)$ is induced by evaluation at
  $0$. Then $F$ defines a functor $\CatIC \to \CatIC$.
\end{defn}

We will call the projection $F(p) \colon
\mathcal{E}\times_\mathcal{C}\mathcal{C}^{\Delta^1}\to\mathcal{C}$ the
\emph{free Cartesian fibration} on $p \colon
\mathcal{E}\to\mathcal{C}$ --- the results of this section will
justify this terminology.

\begin{ex}
  The free Cartesian fibration on the identity
  $\mathcal{C}\to\mathcal{C}$ is the source fibration $F \colon
  \mathcal{C}^{\Delta^1}\to\mathcal{C}$, given by evaluation at
  $0\in\Delta^1$.
\end{ex}

\begin{lemma}
  The functor $F$ factors through the subcategory $\CatICc \to \CatIC$.
\end{lemma}
\begin{proof}
  By \cite{HTT}*{Corollary 2.4.7.12} the projection $F(p) \to
  \mathcal{C}$ is a Cartesian fibration for any $p \colon \mathcal{E}
  \to \mathcal{C}$, and a
  morphism in $F(p)$ is Cartesian \IFF{} its image in $\mathcal{E}$ is
  an equivalence. It is thus clear that for any map $\phi \colon \mathcal{E} \to
  \mathcal{F}$ in $\CatIC$, the induced map $F(\phi)$ preserves Cartesian
  morphisms, since the diagram
  \nolabelcsquare{\mathcal{E} \times_{\mathcal{C}}
    \mathcal{C}^{\Delta^{1}}}{\mathcal{F} \times_{\mathcal{C}}
    \mathcal{C}^{\Delta^{1}}}{\mathcal{E}}{\mathcal{F}} commutes.  
\end{proof}
\begin{remark}
  If $p \colon \mathcal{E} \to \mathcal{C}$ is a functor, the objects
  of $F(p)$ can be identified with pairs $(e, \phi \colon c \to
  p(e))$ where $e \in \mathcal{E}$ and $\phi$ is a morphism in
  $\mathcal{C}$. Similarly, a morphism in $F(p)$ can be identified
  with the data of a morphism $\alpha \colon e' \to e$ in
  $\mathcal{C}$ and a commutative diagram
  \nolabelcsquare{c'}{p(e')}{c}{p(e).}
  If $(e, \phi)$ is an object in $F(p)$ and $\psi \colon c' \to c$ is
  a morphism in $\mathcal{C}$, the Cartesian morphism over $\psi$ with
  target $(e, \phi)$ is the obvious morphism from $(e, \phi\psi)$.
\end{remark}

\begin{thm}\label{thm:free}
  Let $\mathcal{C}$ be an \icat{}. The functor $F \colon \CatIC \to
  \CatICc$ is left adjoint to the forgetful functor $U \colon \CatICc
  \to \CatIC$.
\end{thm}

\begin{remark}
  Analogues of this result in the setting of ordinary categories (as
  well as enriched and internal variants) can be found in
  \cite{StreetFibBicat} and \cite{WeberYoneda}.
\end{remark}

Composition with the degeneracy $s_{0} \colon \Delta^{1} \to
\Delta^{0}$ induces a functor $\mathcal{C} \to
\mathcal{C}^{\Delta^{1}}$ (which sends an object of $\mathcal{C}$ to
the constant functor $\Delta^{1} \to \mathcal{C}$ with that
value). Since the composition of this with both of the evaluation maps
$\mathcal{C}^{\Delta^{1}} \to \mathcal{C}$ is the identity, this
induces a natural map $\mathcal{E} \to \mathcal{E}
\times_{\mathcal{C}} \mathcal{C}^{\Delta^{1}}$ over $\mathcal{C}$,
i.e.\ a natural transformation
\[ \eta \colon \id \to UF \]
of functors $\CatIC \to \CatIC$. We will show that this is a
\emph{unit transformation} in the sense of \cite{HTT}*{Definition
  5.2.2.7}, i.e.\ that it induces an equivalence
\[\MapCc(F(\mathcal{E}), \mathcal{F}) \to \MapC(UF(\mathcal{E}),
U(\mathcal{F})) \to \MapC(\mathcal{E}, U(\mathcal{F}))\] 
for all $\mathcal{E} \to \mathcal{C}$ in $\CatIC$ and $\mathcal{F} \to
\mathcal{C}$ in $\CatICc$.

We first check this for the objects of $\CatIC$ with source
$\Delta^{0}$ and $\Delta^{1}$, which (in a weak sense) generate $\CatIC$ under
colimits. If a map $\Delta^{0}
\to \mathcal{C}$ corresponds to the object $x \in \mathcal{C}$, then its image under $F$
is the projection $\mathcal{C}_{/x} \to \mathcal{C}$. (Strictly
speaking, the image is the ``alternative overcategory''
$\mathcal{C}^{/x}$ in the notation of \cite{HTT}*{\S 4.2.1}, but this
is naturally weakly equivalent to $\mathcal{C}_{/x}$ by
\cite{HTT}*{Proposition 4.2.1.5}.)  Thus in this case we need to show the
following:
\begin{lemma}\label{lem:delta0}\ 
  \begin{enumerate}[(i)]
  \item For every $x \in \mathcal{C}$, the map
    $\MapCc(\mathcal{C}_{/x}, \mathcal{E}) \to \MapC(\{x\},
    \mathcal{E}) \simeq \iota \mathcal{E}_{x}$ is an equivalence.
  \item More generally, for any $\mathcal{X} \in \CatI$, the map
    \[\MapCc( \mathcal{C}_{/x} \times \mathcal{X}, \mathcal{E}) \to
    \MapC( \{x\} \times \mathcal{X}, \mathcal{E}) \simeq
    \Map(\mathcal{X}, \mathcal{E}_{x})\] is an equivalence.
  \end{enumerate}
\end{lemma}
\begin{proof}
  The inclusion of the \icat{} of right fibrations over $\mathcal{C}$
  into $\CatICc$ has a right adjoint, which sends a Cartesian
  fibration $p \colon \mathcal{E} \to \mathcal{C}$ to its restriction to the
  subcategory $\mathcal{E}_{\txt{cart}}$ of $\mathcal{E}$ where the morphisms are the
  $p$-Cartesian morphisms. The map $\MapCc(\mathcal{C}_{/x},
  \mathcal{E}) \to \iota \mathcal{E}_{x}$ thus factors as
  \[ \MapCc(\mathcal{C}_{/x}, \mathcal{E}) \isoto
  \MapC(\mathcal{C}_{/x}, \mathcal{E}_{\txt{cart}}) \to \iota
  \mathcal{E}_{x},\]
  where $\MapC(\mathcal{C}_{/x}, \mathcal{E}_{\txt{cart}})$ is the
  mapping space in the \icat{} of right fibrations over $\mathcal{C}$,
  which is modelled by the contravariant model structure on
  $(\sSet)_{/\mathcal{C}}$ constructed in \cite{HTT}*{\S 2.1.4}.

  By \cite{HTT}*{Proposition 4.4.4.5}, the inclusion $\{x\} \to
  \mathcal{C}_{/x}$ is a trivial cofibration in this model
  category. Since this is a simplicial model category by
  \cite{HTT}*{Proposition 2.1.4.8}, it follows immediately that we
  have an equivalence
  \[ \MapC(\mathcal{C}_{/x}, \mathcal{E}_{\txt{cart}}) \isoto
  \MapC(\{x\}, \mathcal{E}_{\txt{cart}}).\]
  This proves (i). To prove (ii) we simply observe that since the
  model category is simplicial, the product $\{x\} \times K \to\mathcal{C}_{/x} \times K$ is also a trivial cofibration for any
  simplicial set $K$.
\end{proof}

For the case of maps $\Delta^{1} \to \mathcal{C}$, the key observation is:
\begin{propn}\label{propn:pushout}
  If $\Delta^{1} \to \mathcal{C}$ corresponds to a map $f \colon x \to y$ in
  $\mathcal{C}$, then the diagram
  \nolabelcsquare{\mathcal{C}_{/x}}{\mathcal{C}_{/x} \times \Delta^{1}}{\mathcal{C}_{/y}}{\mathcal{C}^{\Delta^{1}} \times_{\mathcal{C}} \Delta^{1}}
  is a pushout square in $\CatICc$, where the top map is induced by
  the inclusion $\{0\} \hookrightarrow \Delta^{1}$.
\end{propn}

\begin{proof}
  Since colimits in $\CatICc \simeq \Fun(\mathcal{C}^{\op}, \CatI)$
  are detected fibrewise, it suffices to show that for every $c \in
  \mathcal{C}$, the diagram on fibres is a pushout in $\CatI$. This
  diagram can be identified with \nolabelcsquare{\MapC(c,
    x)}{\MapC(c, x) \times \Delta^{1}}{\MapC(c, y)}{\mathcal{C}_{c/}
    \times_{\mathcal{C}} \Delta^{1}.}  This is a pushout by
  Lemma~\ref{lem:Dncocartpo}, since $\mathcal{C}_{c/}
  \times_{\mathcal{C}} \Delta^{1} \to \Delta^{1}$ is the left
  fibration corresponding to the map of spaces $\MapC(c, x) \to
  \MapC(c, y)$ induced by composition with $f$.
\end{proof}

\begin{cor}\label{cor:delta1}
  For every map $\sigma \colon \Delta^{1} \to \mathcal{C}$ and every Cartesian
  fibration $\mathcal{E} \to \mathcal{C}$, the map \[\eta_{\sigma}^{*} \colon \MapCc(\mathcal{C}^{\Delta^{1}}
  \times_{\mathcal{C}} \Delta^{1}, \mathcal{E}) \to \MapC(\Delta^{1}, \mathcal{E})\] is an
  equivalence.
\end{cor}
\begin{proof}
  By Proposition~\ref{propn:pushout}, if the map $\sigma$ corresponds
  to a morphism $f \colon x \to y$ in $\mathcal{C}$, we have a
  pullback square \nolabelcsquare{\MapCc(\mathcal{C}^{\Delta^{1}}
    \times_{\mathcal{C}} \Delta^{1},
    \mathcal{E})}{\MapCc(\mathcal{C}_{/y},
    \mathcal{E})}{\MapCc(\mathcal{C}_{/x} \times \Delta^{1},
    \mathcal{E})}{\MapCc(\mathcal{C}_{/x}, \mathcal{E}).}  The map
  $\eta_{\sigma}^{*}$ fits in an obvious map of commutative squares
  from this to the square \nolabelcsquare{\MapC(\Delta^{1},
    \mathcal{E})}{\MapC(\{y\}, \mathcal{E})}{\MapC(\{x\} \times \Delta^{1}, \mathcal{E})}{\MapC(\{x\}, \mathcal{E}),} where the right
  vertical map is given by composition with Cartesian morphims over
  $f$.  Since $\mathcal{E} \to \mathcal{C}$ is a Cartesian fibration,
  this is also a pullback square (this amounts to saying morphisms in
  $\mathcal{E}$ over $f$ are equivalent to composites of a morphism in
  $\mathcal{E}_{x}$ with a Cartesian morphism over $f$). But now, by
  Lemma~\ref{lem:delta0}, we have a natural transformation of pullback
  squares that's an equivalence everywhere except the top left corner,
  so the map in that corner is an equivalence too.
\end{proof}

To complete the proof, we now only need to observe that $F$ preserves
colimits:
\begin{lemma}\label{lem:colim}
  $F$ preserves colimits.
\end{lemma}
\begin{proof}
  Colimits in $\CatICc$ are detected fibrewise, so we need to show
  that for every $x \in \mathcal{C}$, the functor $\mathcal{C}_{x/}
  \times_{\mathcal{C}} (\blank) \colon \CatIC \to \CatI$ preserves
  colimits. But $\mathcal{C}_{x/} \to \mathcal{C}$ is a flat fibration
  by \cite{HA}*{Example B.3.11}, so pullback along it preserves
  colimits as a functor $\CatIC \to \Cat_{\infty/\mathcal{C}_{x/}}$
  (since on the level of model categories the pullback functor is a
  left Quillen functor by \cite{HA}*{Corollary B.3.15}), and the
  forgetful functor $\Cat_{\infty/\mathcal{C}_{x/}} \to \CatI$ also
  preserves colimits.
\end{proof}

\begin{proof}[Proof of Theorem~\ref{thm:free}]
By Lemma~\ref{lem:colim} the source and target of the natural map
\[\MapCc(F(\mathcal{E}), \mathcal{F}) \to \MapC(\mathcal{E},
U(\mathcal{F}))\]
both take colimits in $\mathcal{E}$ to limits of spaces. Since $\CatI$
is a localization of $\mathcal{P}(\simp)$, every object of $\CatIC$ is
canonically the colimit of a diagram consisting of objects of the form
$\Delta^{n} \to \mathcal{C}$, so it suffices to show that the map is
an equivalence for such objects. But $\Delta^{n}$ can in turn be
identified with the colimit
$\Delta^{1} \amalg_{\Delta^{0}} \cdots \amalg_{\Delta^{0}}
\Delta^{1}$,
so it suffices to check that the map is an equivalence when
$\mathcal{E} = \Delta^{0}$ and $\Delta^{1}$. Thus the result follows
from Lemma~\ref{lem:delta0} and Corollary~\ref{cor:delta1}.
\end{proof}

\begin{propn}\label{propn:freeonfuncat}\ 
  \begin{enumerate}[(i)]
  \item Suppose $X \to S$ is a map of \icats{} and $K$ is an
    \icat{}. Then there is a natural equivalence $F( K \times X) \simeq
    K \times F(X)$.
  \item The unit map $X\to F(X)$ induces an equivalence of \icats{}
    \[ \Fun^{\txt{cart}}_{S}(F(X), Y)\isoto\Fun_{S}(X, Y).\]
  \end{enumerate}
\end{propn}
\begin{proof}
  (i) is immediate from the definition of $F$. Then (ii) follows from
  the natural equivalence
  \[
  \begin{split}
\Map(K, \Fun_{S}(A, B)) & \simeq \Map_{S}(K \times A, B) \simeq
  \Map_{S}^{\txt{cart}}(F(K \times A), B) \\ & \simeq
  \Map_{S}^{\txt{cart}}(K \times F(A), B) \simeq \Map(K,
  \Fun_{S}^{\txt{cart}}(F(A), B)).\qedhere
  \end{split}\]
\end{proof}

\section{Natural Transformations as an End}\label{sec:nattrend}
It is a familiar result from ordinary category theory that for two
functors $F, G \colon \mathbf{C} \to \mathbf{D}$ the set of natural
transformations from $F$ to $G$ can be identified with the \emph{end}
of the functor $\mathbf{C}^{\op} \times \mathbf{C} \to \Set$ that
sends $(C, C')$ to $\Hom_{\mathbf{D}}(F(C), G(C'))$. Our goal in this
section is to prove the analogous result for \icats{}:

\begin{propn}\label{propn:nattrend}
  Let $F, G \colon \mathcal{C} \to \mathcal{D}$ be two functors of
  \icats{}. Then the space $\Map_{\Fun(\mathcal{C}, \mathcal{D})}(F,
  G)$ of natural transformations from $F$ to $G$ is naturally
  equivalent to the end of the functor
  \[ \mathcal{C}^{\op} \times \mathcal{C} \xto{(F^{\op}, G)}
  \mathcal{D}^{\op} \times \mathcal{D} \xto{\Map_{\mathcal{D}}}
  \mathcal{S}. \] 
\end{propn}

A proof of this is also given in \cite{GlasmanTHHHodge}*{Proposition
  2.3}; we include a slightly different proof for completeness.

\begin{lemma}\label{lem:compff}
  Suppose $i \colon \mathcal{C}_{0} \hookrightarrow \mathcal{C}$ is a
  fully faithful functor of \icats{}. Then for any \icat{}
  $\mathcal{X}$ the functor $\Fun(\mathcal{X}, \mathcal{C}_{0}) \to
  \Fun(\mathcal{X}, \mathcal{C})$ is also fully faithful.
\end{lemma}
\begin{proof}
  A functor $G \colon \mathcal{A} \to \mathcal{B}$ is fully faithful
  \IFF{} the commutative square of spaces
  \nolabelcsquare{\Map(\Delta^{1}, \mathcal{A})}{\Map(\Delta^{1},
    \mathcal{B})}{\iota \mathcal{A}^{\times 2}}{\iota
    \mathcal{B}^{\times 2}} is Cartesian. Thus, we must show that for
  any $\mathcal{X}$, the square \nolabelcsquare{\Map(\Delta^{1} \times
    \mathcal{X}, \mathcal{C}_{0})}{\Map(\Delta^{1} \times \mathcal{X},
    \mathcal{C})}{\Map(\mathcal{X}, \mathcal{C}_{0})^{\times
      2}}{\Map(\mathcal{X}, \mathcal{C})^{\times 2}} is Cartesian. But
  this is equivalent to the commutative square of \icats{}
  \nolabelcsquare{\mathcal{C}_{0}^{\Delta^{1}}}{\mathcal{C}^{\Delta^{1}}}{\mathcal{C}_{0}^{\times
      2}}{\mathcal{C}^{\times 2}} being Cartesian. By
  \cite{HTT}*{Corollary 2.4.7.11} the vertical maps in this diagram
  are bifibrations in the sense of \cite{HTT}*{Definition 2.4.7.2}, so
  by \cite{HTT}*{Propositions 2.4.7.6 and 2.4.7.7} to prove that this
  square is Cartesian it suffices to show that for all $x,y \in
  \mathcal{C}_{0}$ the induced map on fibres
  $(\mathcal{C}_{0}^{\Delta^{1}})_{(x,y)} \to
  (\mathcal{C}^{\Delta^{1}})_{(ix,iy)}$ is an equivalence. But this
  can be identified with the map $\Map_{\mathcal{C}_{0}}(x, y) \to
  \Map_{\mathcal{C}}(ix, iy)$, which is an equivalence as $i$ is by
  assumption fully faithful.
\end{proof}

\begin{proof}[Proof of Proposition~\ref{propn:nattrend}]
  By \cite{HTT}*{Corollary 3.3.3.4}, we can identify the limit of the
  functor
  \[ \phi \colon \txt{Tw}(\mathcal{C})^{\op} \to \mathcal{C}^{\op} \times
  \mathcal{C} \xto{(F^{\op}, G)} \mathcal{D}^{\op} \times \mathcal{D}
  \xto{\Map_{\mathcal{D}}} \mathcal{S} \] with the space of sections
  of the corresponding left fibration. By
  \cite{HA}*{Proposition 5.2.1.11}, the left fibration classified by
  $\Map_{\mathcal{D}}$ is the projection $\txt{Tw}(\mathcal{D})^{\op} \to
  \mathcal{D}^{\op} \times \mathcal{D}$,
  so the left fibration classified by $\phi$ is the pullback of this
  along $\txt{Tw}(\mathcal{C})^{\op} \to \mathcal{C}^{\op} \times
  \mathcal{C} \to \mathcal{D}^{\op} \times \mathcal{D}$. Thus the
  space of sections is equivalent to the space of commutative diagrams
  \nolabelcsquare{\txt{Tw}(\mathcal{C})^{\op}}{\txt{Tw}(\mathcal{D})^{\op}}{\mathcal{C}^{\op}
    \times \mathcal{C}}{\mathcal{D}^{\op} \times \mathcal{D},}
  i.e.\ the space of maps from $\txt{Tw}(\mathcal{C})$ to the pullback
  of $\txt{Tw}(\mathcal{D})$ in the \icat{} of left fibrations over
  $\mathcal{C}^{\op} \times \mathcal{C}$. Using the ``straightening''
  equivalence between this \icat{} and that of functors
  $\mathcal{C}^{\op} \times \mathcal{C} \to \mathcal{S}$ we can
  identify our limit with the space of maps from $y_{\mathcal{C}}$ to
  $F^{*}\circ y_{\mathcal{D}} \circ G$ in $\Fun(\mathcal{C}^{\op}
  \times \mathcal{C}, \mathcal{S}) \simeq \Fun(\mathcal{C},
  \mathcal{P}(\mathcal{C}))$. Since $F^{*}$ has a left adjoint
  $F_{!}$, we have an equivalence
  \[ \Map_{\Fun(\mathcal{C},
    \mathcal{P}(\mathcal{C}))}(y_{\mathcal{C}},
  F^{*}y_{\mathcal{D}}\circ G) \simeq \Map_{\Fun(\mathcal{C},
    \mathcal{P}(\mathcal{D}))}(F_{!}y_{\mathcal{C}},
  y_{\mathcal{D}}\circ G).\] But by \cite{HTT}*{Proposition 5.2.6.3} the functor
  $F_{!}y_{\mathcal{C}}$ is equivalent to $y_{\mathcal{D}} \circ F$,
  and so the limit is equivalent to $\Map_{\Fun(\mathcal{C},
    \mathcal{P}(\mathcal{D}))}(y_{\mathcal{D}}\circ F, y_{\mathcal{D}}
  \circ G)$. The Yoneda embedding $y_{\mathcal{D}}$ is fully faithful
  by \cite{HTT}*{Proposition 5.1.3.1}, so Lemma~\ref{lem:compff}
  implies that the functor $\Fun(\mathcal{C}, \mathcal{D}) \to
  \Fun(\mathcal{C}, \mathcal{P}(\mathcal{D}))$ given by composition
  with $y_{\mathcal{D}}$ is fully faithful, hence we have an
  equivalence
  \[ \Map_{\Fun(\mathcal{C},
    \mathcal{P}(\mathcal{D}))}(y_{\mathcal{C}}\circ F, y_{\mathcal{D}}
  \circ G) \simeq \Map_{\Fun(\mathcal{C},\mathcal{D})}(F, G),\]
  which completes the proof.
\end{proof}

\section{Enhanced Mapping Functors}\label{sec:enhmap}
The Yoneda embedding for \icats{}, constructed in
\cite{HTT}*{Proposition 5.1.3.1} or \cite{HA}*{Proposition 5.2.1.11},
gives for any \icat{} $\mathcal{C}$ a mapping space functor
$\Map_{\mathcal{C}} \colon \mathcal{C}^{\op} \times \mathcal{C} \to
\mathcal{S}$. In some cases, this is the underlying functor to spaces
of an interesting functor $\mathcal{C}^{\op} \times \mathcal{C} \to
\CatI$ --- in particular, this is the case if $\mathcal{C}$ is the
underlying \icat{} of an $(\infty,2)$-category. For a definition and a comparison of different models of $(\infty,2)$-categories see \cite{LurieGoodwillie}.

\begin{defn}
  A \emph{mapping \icat{} functor} for an \icat{} $\mathcal{C}$ is a
  functor \[\txt{MAP}_{\mathcal{C}} \colon \mathcal{C}^{\op} \times \mathcal{C} \to \CatI\] together with an equivalence from the composite $\mathcal{C}^{\op} \times \mathcal{C} \to \CatI
  \xto{\iota} \mathcal{S}$ to the mapping space functor $\Map_{\mathcal{C}}$.
\end{defn}

\begin{lemma}
  Suppose $\mathcal{C}$ is an $(\infty,2)$-category with underlying
  \icat{} $\mathcal{C}'$. Then $\mathcal{C}'$ has a mapping \icat{}
  functor that sends $(C, D)$ to the \icat{} of maps from $C$ to $D$
  in $\mathcal{C}$.
\end{lemma}
\begin{proof}
  This follows from the same argument as in \cite{HTT}*{\S 5.1.3},
  using the model of $(\infty,2)$-categories as categories enriched in
  marked simplicial sets, cf.\ \cite{LurieGoodwillie}.
\end{proof}

\begin{ex}
  The \icat{} $\CatI$ of \icats{} has a mapping \icat{} functor
  \[\MAP_{\CatI} := \Fun,\] defined using the construction of $\CatI$
  as the coherent nerve of the simplicial category
  of quasicategories.
\end{ex}

\begin{lemma}\label{lem:MAPfuncat}
  Suppose $\mathcal{C}$ is an \icat{} with a mapping \icat{} functor
  $\MAP_{\mathcal{C}}$. Then for any \icat{} $\mathcal{D}$ the functor
  \icat{} $\mathcal{C}^{\mathcal{D}}$ has a mapping \icat{} functor
  $\MAP_{\mathcal{C}^{\mathcal{D}}}$ given by the composite
  \[ (\mathcal{C}^{\mathcal{D}})^{\op} \times
  \mathcal{C}^{\mathcal{D}} \to \Fun(\mathcal{D}^{\op} \times
  \mathcal{D}, \mathcal{C}^{\op} \times \mathcal{C}) \to
  \Fun(\txt{Tw}(\mathcal{D})^{\op}, \CatI) \xto{\lim} \CatI,\] where the
  second functor is given by composition with the projection
  $\txt{Tw}(\mathcal{D})^{\op} \to \mathcal{D}^{\op} \times \mathcal{D}$ and
  $\MAP_{\mathcal{C}}$.
\end{lemma}
\begin{proof}
  We must show that the underlying functor to spaces $\iota \circ
  \MAP_{\mathcal{C}^{\mathcal{D}}}$ is equivalent to
  $\Map_{\mathcal{C}^{\mathcal{D}}}$. Since $\iota$ preserves limits
  (being a right adjoint), this follows immediately from
  Proposition~\ref{propn:nattrend}. 
\end{proof}

\begin{defn}
  Suppose $\mathcal{C}$ is an \icat{} with a mapping \icat{} functor
  $\MAP_{\mathcal{C}}$. We say that $\mathcal{C}$ is \emph{tensored
    over $\CatI$} if for every $C \in \mathcal{C}$ the functor
  $\MAP_{\mathcal{C}}(C, \blank)$ has a left adjoint $\blank \otimes C
  \colon \CatI \to \mathcal{C}$; in this case these adjoints determine
  an essentially unique functor $\otimes \colon \CatI \times
  \mathcal{C} \to \mathcal{C}$.
\end{defn}

\begin{ex}
  The \icat{} $\CatI$ is obviously tensored over $\CatI$ via the
  Cartesian product $\times \colon \CatI \times \CatI \to \CatI$.
\end{ex}

\begin{lemma}\label{lem:tensfuncat}
  Suppose $\mathcal{C}$ is an \icat{} with a mapping \icat{}
  $\MAP_{\mathcal{C}}$ that is tensored over $\CatI$. Then for any
  \icat{} $\mathcal{D}$, the mapping \icat{} functor for
  $\mathcal{C}^{\mathcal{D}}$ defined in Lemma~\ref{lem:MAPfuncat} is
  also tensored over $\CatI$, via the composite
  \[ \CatI \times \mathcal{C}^{\mathcal{D}} \to \CatI^{\mathcal{D}}
  \times \mathcal{C}^{\mathcal{D}} \simeq (\CatI \times
  \mathcal{C})^{\mathcal{D}} \to \mathcal{C}^{\mathcal{D}} \]
  where the first functor is given by composition with the functor
  $\mathcal{D} \to *$ and the last by composition with the tensor
  functor for $\mathcal{C}$.
\end{lemma}
\begin{proof}
  We must show that for every functor $F \colon \mathcal{D} \to
  \mathcal{C}$ there is a natural equivalence
  \[ \Map_{\mathcal{C}^{\mathcal{D}}}(\mathcal{X} \otimes F, G) \simeq
  \Map_{\CatI}(\mathcal{X}, \MAP_{\mathcal{C}^{\mathcal{D}}}(F, G)).\]
  By Proposition~\ref{propn:nattrend} and the definition of $\otimes$ for $\mathcal{C}^{\mathcal{D}}$, there is a natural equivalence
  \[ \Map_{\mathcal{C}^{\mathcal{D}}}(\mathcal{X} \otimes F, G) \simeq
  \lim_{\txt{Tw}(\mathcal{D})^{\op}} \Map_{\mathcal{C}}(\mathcal{X} \otimes
  F(\blank), G(\blank)).\] Now using that $\mathcal{C}$ is tensored
  over $\CatI$, this is naturally equivalent to
  \[\lim_{\txt{Tw}(\mathcal{D})^{\op}} \Map_{\CatI}(\mathcal{X},
  \MAP_{\mathcal{C}}(F(\blank), G(\blank))).\] Moving the limit inside,
  this is \[\Map_{\CatI}(\mathcal{X}, \lim_{\txt{Tw}(\mathcal{D})^{\op}}
  \MAP_{\mathcal{C}}(F(\blank), G(\blank))),\] which is
  $\Map_{\CatI}(\mathcal{X}, \MAP_{\mathcal{C}^{\mathcal{D}}}(F, G))$
  by definition.
\end{proof}

\begin{ex}
  For any \icat{} $\mathcal{D}$, the \icat{} $\CatI^{\mathcal{D}}$ is
  tensored over $\CatI$: $\mathcal{X} \otimes F$ is the functor $D
  \mapsto \mathcal{X} \times F(D)$.
\end{ex}

In the case where $\mathcal{C}$ is the \icat{} $\CatI$ of \icats{},
Lemma~\ref{lem:MAPfuncat} gives a mapping \icat{} functor
\[\txt{Nat}_{\mathcal{D}^{\op}} := \MAP_{\Fun(\mathcal{D}^{\op},
  \CatI)}\] for $\Fun(\mathcal{D}^{\op}, \CatI)$, for any \icat{}
$\mathcal{D}$. However, using the equivalence $\Fun(\mathcal{D}^{\op},
\CatI) \simeq \CatIDc$ we can construct another such
functor: the space of maps from $\mathcal{E}$ to $\mathcal{E}'$ in
$\CatIDc$ is the underlying $\infty$-groupoid of the
\icat{} $\Fun^{\txt{cart}}_{\mathcal{D}}(\mathcal{E}, \mathcal{E}')$,
the full subcategory of $\Fun_{\mathcal{D}}(\mathcal{E},
\mathcal{E}')$ spanned by the functors that preserve Cartesian
morphisms. We will now prove that these two functors are equivalent:

\begin{propn}\label{propn:CartNateq}
  For every \icat{} $\mathcal{C}$ there is a natural equivalence
  \[\Fun^{\txt{cart}}_{\mathcal{C}}(\mathcal{E}, \mathcal{E}') \simeq
  \txt{Nat}_{\mathcal{C}^{\op}}(\txt{St}_{\mathcal{C}}\mathcal{E},
  \txt{St}_{\mathcal{C}}\mathcal{E}').\]
\end{propn}

\begin{proof}
  By the Yoneda Lemma it suffices to show that there are natural
  equivalences \[\Map_{\CatI}(\mathcal{X},
  \Fun^{\txt{cart}}_{\mathcal{C}}(\mathcal{E}, \mathcal{E}')) \simeq
  \Map_{\CatI}(\mathcal{X},
  \txt{Nat}_{\mathcal{C}^{\op}}(\txt{St}_{\mathcal{C}}\mathcal{E},
  \txt{St}_{\mathcal{C}}\mathcal{E}')).\] It is easy to see that
  $\Map_{\CatI}(\mathcal{X},
  \Fun^{\txt{cart}}_{\mathcal{C}}(\mathcal{E}, \mathcal{E}'))$ is
  naturally equivalent to $\Map_{\CatICc}(\mathcal{X} \times
  \mathcal{E}, \mathcal{E}')$ --- these correspond to the same
  components of $\Map_{\CatI}(\mathcal{X},
  \Fun_{\mathcal{C}}(\mathcal{E}, \mathcal{E}'))$. The equivalence
  $\txt{St}_{\mathcal{C}}$ preserves products, so this is equivalent
  to the mapping space \[\Map_{\Fun(\mathcal{C}^{\op},
    \CatI)}(\txt{St}_{\mathcal{C}}(\mathcal{X} \times \mathcal{C})
  \times \txt{St}_{\mathcal{C}}\mathcal{E},
  \txt{St}_{\mathcal{C}}\mathcal{E}').\] But the projection
  $\mathcal{X} \times \mathcal{C} \to \mathcal{C}$ corresponds to the
  constant functor $c^{*}\mathcal{X} \colon \mathcal{C}^{\op} \to *
  \to \CatI$ with value $\mathcal{C}$ (since the Cartesian fibration
  classified by this composite is precisely the pullback of
  $\mathcal{X} \to *$ along $\mathcal{C} \to *$). Thus there is a
  natural equivalence
  \[\Map_{\CatI}(\mathcal{X},
  \Fun^{\txt{cart}}_{\mathcal{C}}(\mathcal{E}, \mathcal{E}')) \simeq
  \Map_{\Fun(\mathcal{C}^{\op}, \CatI)}(c^{*}\mathcal{X} \times
  \txt{St}_{\mathcal{C}}\mathcal{E},
  \txt{St}_{\mathcal{C}}\mathcal{E}').\] But by
  Lemma~\ref{lem:tensfuncat}, the \icat{} $\Fun(\mathcal{C}^{\op},
  \CatI)$ is tensored over $\CatI$ and this is naturally equivalent
  to $\Map_{\CatI}(\mathcal{X},
  \txt{Nat}_{\mathcal{C}^{\op}}(\txt{St}_{\mathcal{C}}\mathcal{E},
  \txt{St}_{\mathcal{C}}\mathcal{E}'))$, as required.
\end{proof}

\section{Cartesian and CoCartesian Fibrations as Weighted Colimits}\label{sec:Cartwtcolim}
In ordinary category theory it is a familiar fact that the
Grothendieck fibration classified by a functor $F \colon
\mathbf{C}^{\op} \to \Cat$ can be identified with the lax colimit of
$F$, and the Grothendieck opfibration classified by a functor $F
\colon \mathbf{C} \to \Cat$ with the oplax colimit of $F$. In this
section we will show that Cartesian and coCartesian fibrations admit
analogous descriptions.

It is immediate from our formula for the free Cartesian fibration
that the sections of a Cartesian fibration are given by the oplax
limit of the corresponding functor:
\begin{propn}\label{propn:seclaxlim}
  The \icat{} of sections of the Cartesian fibration classified by $F$
  is given by the oplax limit of $F$. In other words, there is a
  natural equivalence \[ \Fun_{\mathcal{C}}(\mathcal{C},
  \txt{Un}_{\mathcal{C}}(F)) \simeq \lim_{\txt{Tw}(\mathcal{C})^\op}
  \Fun(\mathcal{C}_{\blank/}, F(\blank))\] of functors
  $\Fun(\mathcal{C}^{\op}, \CatI) \to \CatI$.
\end{propn}
\begin{proof}
  By Theorem~\ref{thm:free} and Proposition~\ref{propn:CartNateq} we
  have natural equivalences
  \[\displaylines{ \Fun_{\mathcal{C}}(\mathcal{C}, \txt{Un}_{\mathcal{C}}(F)) \simeq
  \Fun_{\mathcal{C}}^{\txt{cart}}(F(\mathcal{C}),  \txt{Un}_{\mathcal{C}}(F)) \simeq\hfill\cr
\hfill\simeq  \txt{Nat}_{\mathcal{C}^{\op}}(\mathcal{C}_{\blank/}, F) \simeq
  \lim_{\txt{Tw}(\mathcal{C})^{\op}}\Fun(\mathcal{C}_{\blank/},
  F(\blank)).\qedhere}\]
\end{proof}

\begin{defn}
  Let $F \colon \mathcal{C} \to \CatI$ be a functor, and let
  $\mathcal{F} \to \mathcal{C}$ be its associated coCartesian
  fibration. Given an \icat{} $\mathcal{X}$, write
  $\Phi_{\mathcal{X}}^{F}$ for the simplicial set over $\mathcal{C}$
  with the universal property
  \[ \Hom(K \times_{\mathcal{C}} \mathcal{F}, \mathcal{X}) \cong
  \Hom_{\mathcal{C}}(K, \Phi^{F}_{\mathcal{X}}).\] By
  \cite{HTT}*{Corollary 3.2.2.13} the projection
  $\Phi_{\mathcal{X}}^{F} \to \mathcal{C}$ is a Cartesian fibration.
\end{defn}

\begin{propn}\label{propn:UnFunrightadj}
  The Cartesian fibration $\Phi_{\mathcal{X}}^{F} \to \mathcal{C}$
  corresponds to the functor \[\Fun(F(\blank), \mathcal{X}) \colon
  \mathcal{C}^{\op} \to \CatI.\]
\end{propn}
\begin{proof}
  We first consider the case where $\mathcal{C}$ is a simplex
  $\Delta^{n}$. By Proposition~\ref{propn:sxweighted} there are natural equivalences
  \[\colim_{\txt{Tw}([n])} \phi(\blank) \times [n]_{\blank/} \isoto
  \txt{Un}^{\txt{co}}_{[n]}(\phi)\] for any $\phi \colon [n] \to \CatI$, natural in $\simp^{\op}$. Thus by
  Proposition~\ref{propn:seclaxlim} there are natural equivalences
  \[ 
  \begin{split} 
    \Fun_{\Delta^{n}}(\Delta^{n}, \Phi^{\phi}_{\mathcal{X}}) & \simeq
  \Fun(\txt{Un}^{\txt{co}}_{[n]}(\phi), \mathcal{X}) \simeq
  \lim_{\txt{Tw}([n])^{\op}} \Fun([n]_{\blank/}, \Fun(\phi(\blank),
  \mathcal{X})) \\ & \simeq \Fun_{\Delta^{n}}(\Delta^{n},
  \txt{Un}_{[n]}(\Fun(\phi(\blank), \mathcal{X}))).
\end{split}
\]
  Since this equivalence is natural in $\simp^{\op}$ and $\CatI$ is a
  localization of the presheaf \icat{} $\mathcal{P}(\simp)$, we get by
  the Yoneda lemma a natural equivalence
  \[ \Phi^{\phi}_{\mathcal{X}} \simeq
  \txt{Un}_{[n]}^{\txt{cart}}(\Fun(\phi(\blank), \mathcal{X}))).\]
  Since $\CatI$ is an accessible localization of $\mathcal{P}(\simp)$,
  any \icat{} $\mathcal{C}$ is naturally equivalent to the colimit of
  the diagram $\simp^{\op}_{/\mathcal{C}} \to \CatIC \to \CatI$. Now given $F \colon \mathcal{C} \to \CatI$, we have, since pullback
  along a Cartesian fibration preserves colimits,
  \[ \txt{Un}_{\mathcal{C}}^{\txt{cart}}(\Fun(F(\blank),
  \mathcal{X}))) \simeq \colim_{\sigma \in \simp^{\op}_{/\mathcal{C}}}
  \txt{Un}_{[n]}(\Fun(F\sigma(\blank), \mathcal{X})) \simeq
  \colim_{\sigma \in \simp^{\op}_{/\mathcal{C}}}
  \Phi^{F\sigma}_{\mathcal{X}} \simeq \Phi^{F}_{\mathcal{X}},\] which
  completes the proof.
\end{proof}

\begin{thm}\label{thm:Uncolax}
  The coCartesian fibration classified by a functor $F \colon
  \mathcal{C} \to \CatI$ is given by the oplax colimit of $F$. In other
  words, there is a natural
  equivalence \[\txt{Un}^{\txt{co}}_{\mathcal{C}}(F) \simeq
  \colim_{\txt{Tw}(\mathcal{C})} F(\blank) \times \mathcal{C}_{\blank/}
  \] of functors $\Fun(\mathcal{C}, \CatI) \to \CatI$.
\end{thm}
\begin{proof}
  Let $F \colon \mathcal{C} \to \CatI$ be a functor. Then by
  Proposition~\ref{propn:UnFunrightadj}, we have a natural equivalence
  \[ \Fun(\txt{Un}_{\mathcal{C}}(F), \mathcal{X}) \simeq
  \Fun_{\mathcal{C}}(\mathcal{C}, \Phi^{F}_{\mathcal{X}}).\] By
  Proposition~\ref{propn:seclaxlim} we have a natural equivalence
  between the right-hand side and \[\lim_{\txt{Tw}(\mathcal{C})^\op}
  \Fun(\mathcal{C}_{\blank/}, \Fun(F(\blank), \mathcal{X})) \simeq
  \Fun\big(\colim_{\txt{Tw}(\mathcal{C})} F(\blank)\times \mathcal{C}_{\blank/}, \mathcal{X}\big)
  .\] By the Yoneda lemma, it follows that
  $\txt{Un}_{\mathcal{C}}(F)$ is naturally equivalent to
  $\colim_{\txt{Tw}(\mathcal{C})}F(\blank) \times
  \mathcal{C}_{\blank/}$.
\end{proof}

\begin{cor}
  Any \icat{} $\mathcal{C}$ is the oplax colimit of the constant functor
  $\mathcal{C} \to \CatI$ with value $*$.
\end{cor}
\begin{proof}
  The identity $\mathcal{C} \to \mathcal{C}$ is the coCartesian
  fibration classified by this functor.
\end{proof}

\begin{cor}
  The Cartesian fibration classified by a functor $F \colon
  \mathcal{C}^{\op} \to \CatI$ is given by the lax colimit of
  $F$. In other words, there is a natural
  equivalence \[\txt{Un}_{\mathcal{C}}(F) \simeq
  \colim_{\txt{Tw}(\mathcal{C^\op})} F(\blank) \times \mathcal{C}_{/\blank}
  \] of functors $\Fun(\mathcal{C}^{\op}, \CatI) \to \CatI$.
\end{cor}
\begin{proof}
  We have a natural equivalence $\txt{Un}_{\mathcal{C}}(F) \simeq
  \txt{Un}_{\mathcal{C}^{\op}}^{\txt{co}}(F^\op)^{\op}$. Since
  $(\blank)^{\op}$ is an automorphism of $\CatI$ it preserves
  colimits, so by Theorem~\ref{thm:Uncolax} we have
  \[ \txt{Un}_{\mathcal{C}}(F) \simeq \left(\colim_{\txt{Tw}(\mathcal{C}^\op)}
  \,F(\blank)^\op \times (\mathcal{C}^{\op})_{\blank/} \right)^{\op} \simeq
\colim_{\txt{Tw}(\mathcal{C}^\op)} F(\blank) \times \mathcal{C}_{/\blank}  .\qedhere\]
\end{proof}

Similarly, dualizing Proposition~\ref{propn:seclaxlim} gives:
\begin{cor}
  The \icat{} of sections of the coCartesian fibration classified by
  $F$ is given by the lax limit of $F$. In other words, there is a
  natural equivalence
  \[ \Fun_{\mathcal{C}}(\mathcal{C},
  \txt{Un}^{\txt{co}}_{\mathcal{C}}(F)) \simeq
  \lim_{\txt{Tw}(\mathcal{C})^\op} \Fun(\mathcal{C}_{/\blank},
  F(\blank))\] of functors $\Fun(\mathcal{C}, \CatI) \to \CatI$.
\end{cor}

\section{2-Representable Functors}\label{sec:laxrep}
Suppose $\mathcal{C}$ is an \icat{} equipped with a mapping \icat{}
functor $\MAP_{\mathcal{C}}\colon\mathcal{C}^{\op}\times\mathcal{C}\to\CatI$. We say a functor $F \colon
\mathcal{C}^{\op} \to \CatI$ is \emph{2-representable} by $C \in
\mathcal{C}$ if $F$ is equivalent to $\MAP_{\mathcal{C}}(\blank,
C)$. Similarly, we say a Cartesian fibration $p \colon \mathcal{E} \to
\mathcal{C}$ is \emph{2-representable} by $C$ if $p$ is classified
by the functor $\MAP_{\mathcal{C}}(\blank, C)$. Our goal in this
section is to prove that if a Cartesian fibration $p \colon
\mathcal{E} \to \mathcal{C}$ is 2-representable then, under mild
hypotheses, the same is true for the induced map
$\mathcal{E}^{\mathcal{D}}\to \mathcal{C}^{\mathcal{D}}$ for any \icat{}
$\mathcal{D}$. We begin by giving a somewhat unwieldy description of
the functor classifying such fibrations for an arbitrary Cartesian
fibration $p$:

\begin{propn}\label{propn:ftrcart}
  Suppose $p \colon \mathcal{E} \to \mathcal{C}$ is a Cartesian
  fibration corresponding to a functor $F \colon \mathcal{C}^{\op} \to
  \CatI$. Then for any \icat{} $\mathcal{D}$ the functor
  $\mathcal{E}^{\mathcal{D}} \to \mathcal{C}^{\mathcal{D}}$ given by
  composition with $p$ is a Cartesian fibration
  classified by a functor $F_{\mathcal{D}} \colon
  (\mathcal{C}^{\mathcal{D}})^{\op} \to \CatI$ that sends a functor
  $\phi \colon \mathcal{D} \to \mathcal{C}$ to
  \[ \lim_{\txt{Tw}(\mathcal{D})^\op} \Fun(\mathcal{D}_{\blank/},
  F\circ\phi(\blank)).\]
\end{propn}

\begin{proof}
  The induced functor $\mathcal{E}^{\mathcal{D}} \to
  \mathcal{C}^{\mathcal{D}}$ is a Cartesian fibration by
  \cite[Proposition 3.1.2.1]{HTT}.
  For $f \colon K \times \mathcal{D}\to \mathcal{C}$ we have a natural
  equivalence
  \[ \Fun_{\mathcal{C}^{\mathcal{D}}}(K, \mathcal{E}^{\mathcal{D}})
  \simeq \Fun_{\mathcal{C}}(K \times \mathcal{D}, \mathcal{E}) \simeq
  \Fun_{K \times \mathcal{D}}(K \times \mathcal{D},
  f^{*}\mathcal{E}).\]
  But then by Proposition~\ref{propn:seclaxlim} we have a natural equivalence
  \[ \Fun_{K \times \mathcal{D}}(K \times \mathcal{D},
  f^{*}\mathcal{E}) \simeq \lim_{\txt{Tw}(K \times \mathcal{D})^{\op}} \Fun((K \times
  \mathcal{D})_{\blank/}, F\circ f(\blank)),\]
  and then as $\txt{Tw}$ preserves products (being a right adjoint) we
  can rewrite this as
  \[ \lim_{\txt{Tw}(K)^{\op}} \Fun(K_{\blank/}, \lim_{\txt{Tw}(\mathcal{D})^{\op}}
  \Fun(\mathcal{D}_{\blank/}, F\circ f(\blank))) \simeq
  \lim_{\txt{Tw}(K)^\op}\Fun(K_{\blank/}, F_{\mathcal{D}}\circ
  f(\blank)), \]
  which we can identify, using Proposition~\ref{propn:seclaxlim}, with
  \[\Fun_{K}(K,
  f^{*}\txt{Un}_{\mathcal{C}^{\mathcal{D}}}(F_{\mathcal{D}})) \simeq
  \Fun_{\mathcal{C}^{\mathcal{D}}}(K,
  \txt{Un}_{\mathcal{C}^{\mathcal{D}}}(F_{\mathcal{D}})).\]
  Now applying the Yoneda Lemma completes the proof.
\end{proof}

\begin{defn}
  Suppose $\mathcal{C}$ is an \icat{} equipped with a mapping \icat{}
  functor. We say that $\mathcal{C}$ is \emph{cotensored over $\CatI$}
  if for every $C \in \mathcal{C}$ the functor
  $\MAP_{\mathcal{C}}(\blank, C) \colon \mathcal{C} \to \CatI^{\op}$
  has a right adjoint $C^{(\blank)} \colon \CatI^{\op} \to \mathcal{C}$;
  in this case these adjoints determine an essentially unique functor
  $(\blank)^{(\blank)} \colon \CatI^{\op} \times \mathcal{C} \to
  \mathcal{C}$. If $\mathcal{C}$ is also tensored over $\CatI$, being
  cotensored is equivalent to the functor $\blank \otimes \mathcal{X}$
  having a right adjoint $(\blank)^{\mathcal{X}}$ for all \icats{}
  $\mathcal{X}$.
\end{defn}

\begin{cor}
  Let $\mathcal{C}$ be an \icat{} equipped with a mapping \icat{} functor
  $\MAP_{\mathcal{C}}$ that is tensored and cotensored over
  $\CatI$, and suppose $p \colon \mathcal{E} \to \mathcal{C}$ is a
  Cartesian fibration that is 2-representable by $C \in
  \mathcal{C}$. Then for any \icat{} $\mathcal{D}$ the fibration
  $\mathcal{E}^{\mathcal{D}} \to \mathcal{C}^{\mathcal{D}}$ is 2-representable by the functor $C^{\mathcal{D}_{\blank/}}$.
\end{cor}
\begin{proof}
  By Proposition~\ref{propn:ftrcart} this fibration corresponds to the
  functor sending $\phi \colon \mathcal{D} \to \mathcal{C}$ to
  \[ \lim_{\txt{Tw}(\mathcal{D})^{\op}} \Fun(\mathcal{D}_{\blank/},
  F\circ\phi(\blank)),\]
  where $F$ is the functor corresponding to $p$. If $F$ is 2-representable by $C$, this is equivalent to 
  \[ \lim_{\txt{Tw}(\mathcal{D})^{\op}} \Fun(\mathcal{D}_{\blank/},
  \MAP_{\mathcal{C}}(\phi(\blank), C) \simeq
  \lim_{\txt{Tw}(\mathcal{D})^{\op}} \MAP_{\mathcal{C}}(\phi(\blank),
  C^{\mathcal{D}_{\blank/}}),\]
  which is $\MAP_{\mathcal{C}^{\mathcal{D}}}(\phi,
  C^{\mathcal{D}_{\blank/}})$ by definition of the mapping \icat{} of
  $\mathcal{C}^{\mathcal{D}}$.
\end{proof}

\section{Some Cartesian Fibrations Identified}\label{sec:proof}
In this section we will use our results so far to explicitly identify
the Cartesian fibrations classified by certain classes of
functors. This is the key input needed to prove our presentability
result in the next section. We start with some notation:

\begin{defn}
  If $p \colon \mathcal{E} \to \mathcal{B}$ is a functor of \icats{},
  we denote by $\mathcal{E}^{\triangleright}_{\mathcal{B}}$ the
  pushout \[\mathcal{B} \amalg_{\mathcal{E} \times \{1\}} \mathcal{E}
  \times \Delta^{1},\]
  and by $\mathcal{E}^{\triangleleft}_{\mathcal{B}}$ the
  pushout \[\mathcal{B} \amalg_{\mathcal{E} \times \{0\}} \mathcal{E}
  \times \Delta^{1}.\]
\end{defn}

\begin{warning}
  The notations $\mathcal{E}^{\triangleright}_{\mathcal{B}}$ and
  $\mathcal{E}^{\triangleleft}_{\mathcal{B}}$ are somewhat abusive, as
  these simplicial sets depend on the functor $p$ rather than just on
  $\mathcal{E}$ and $\mathcal{B}$. Moreover, if $\mathcal{B} =
  \Delta^{0}$ then the definition does not reduce to
  $\mathcal{E}^{\triangleright}$ and $\mathcal{E}^{\triangleleft}$ but
  rather the ``alternative joins'' $\mathcal{E} \diamond \Delta^{0}$
  and $\Delta^{0} \diamond \mathcal{E}$ in the notation of
  \cite{HTT}*{\S 4.2.1}. However, these are weakly equivalent to the
  usual joins by \cite{HTT}*{Proposition 4.2.1.2}.
\end{warning}

\begin{remark}
  Observe that $(\mathcal{E}_{\mathcal{B}}^{\triangleright})^{\op}
  \simeq (\mathcal{E}^{\op})_{\mathcal{B}^{\op}}^{\triangleleft}$ and
  $(\mathcal{E}_{\mathcal{B}}^{\triangleleft})^{\op} \simeq
  (\mathcal{E}^{\op})_{\mathcal{B}^{\op}}^{\triangleright}$.
\end{remark}

We then have the following simple observation:
\begin{lemma}\label{lem:PEBident}
  Given a functor $p \colon \mathcal{E} \to \mathcal{B}$, write
  $i \colon \mathcal{B} \hookrightarrow
  \mathcal{E}_{\mathcal{B}}^{\triangleleft}$
  and
  $j \colon \mathcal{B} \hookrightarrow
  \mathcal{E}_{\mathcal{B}}^{\triangleright}$
  for the inclusions in the pushout diagrams defining
  $\mathcal{E}_{\mathcal{B}}^{\triangleleft}$ and
  $\mathcal{E}_{\mathcal{B}}^{\triangleright}$, respectively. Then for
  any \icat{} $\mathcal{D}$, we have:
  \begin{enumerate}[(i)]
  \item The functor \[i^{*} \colon
    \Fun(\mathcal{E}_{\mathcal{B}}^{\triangleleft}, \mathcal{D}) \to
    \Fun(\mathcal{B}, \mathcal{D})\] given by composition with $i$ is
    a Cartesian fibration classified by the functor $\Fun(\mathcal{B},
    \mathcal{D})^{\op} \to \CatI$ that sends $F$ to $\Fun(\mathcal{E},
    \mathcal{D})_{F\circ p/}$.
  \item The functor \[j^{*} \colon
    \Fun(\mathcal{E}_{\mathcal{B}}^{\triangleright}, \mathcal{D}) \to
    \Fun(\mathcal{B}, \mathcal{D})\] given by composition with $j$ is
    a coCartesian fibration classified by the functor $\Fun(\mathcal{B},
    \mathcal{D}) \to \CatI$ that sends $F$ to $\Fun(\mathcal{E},
    \mathcal{D})_{/F\circ p}$.
  \end{enumerate}
\end{lemma}
\begin{proof}
  We will prove (i); the proof of (ii) is similar. By the definition
  of $\mathcal{E}_{\mathcal{B}}^{\triangleleft}$ we have a pullback
  square \csquare{\Fun(\mathcal{E}_{\mathcal{B}}^{\triangleleft},
    \mathcal{D})}{\Fun(\mathcal{E} \times \Delta^{1},
    \mathcal{D})}{\Fun(\mathcal{B}, \mathcal{D})}{\Fun(\mathcal{E},
    \mathcal{D})}{}{i^{*}}{\txt{ev}_{0}}{p^{*},} where the right
  vertical map can be identified with the evaluation-at-0 functor
  \[
  \Fun(\mathcal{E}, \mathcal{D})^{\Delta^{1}} \to \Fun(\mathcal{E},\mathcal{D})
  \]
 This is the Cartesian fibration classified by the
  undercategory functor $\Fun(\mathcal{E}, \mathcal{D})_{(\blank)/}$,
   hence the pullback $i^{*}$ is the Cartesian fibration classified by
  the composite functor $\Fun(\mathcal{E},
  \mathcal{D})_{p^{*}(\blank)/}$.
\end{proof}
\begin{remark}
  If $\mathcal{D}$ has pushouts, then $i^{*}$ is also a coCartesian
  fibration, with coCartesian morphisms given by taking
  pushouts. Similarly, if $\mathcal{D}$ has pullbacks then $j^{*}$ is
  also a Cartesian fibration, with Cartesian morphisms given by taking
  pullbacks.
\end{remark}

In particular, given a map $p \colon \mathcal{E} \to \mathcal{B}$ we
see that $\mathcal{P}(\mathcal{E}_{\mathcal{B}}^{\triangleleft}) \to
\mathcal{P}(\mathcal{B})$ is a coCartesian and Cartesian fibration (recall that $\mathcal{P}(\mathcal{C}) = \Fun(\mathcal{C}^{\op},\mathcal{S})$ denotes the presheaf $\infty$-category).
The corresponding functors are given on objects by
$\mathcal{P}(\mathcal{E})_{/p^{*}(\blank)}$, with functoriality
determined by composition and pullbacks, respectively. Our next goal
is to give an alternative description of this functor:
\begin{propn}\label{propn:PEBCart}
  Let $p \colon \mathcal{E} \to \mathcal{B}$ be a functor of \icats{},
  and let $j \colon \mathcal{B} \to
  \mathcal{E}^{\triangleleft}_{\mathcal{B}}$ be the obvious
  inclusion.  Then the functor $j^{*} \colon
  \mathcal{P}(\mathcal{E}^{\triangleleft}_{\mathcal{B}}) \to
  \mathcal{P}(\mathcal{B})$ is a Cartesian fibration corresponding to
  the functor $\mathcal{P}(\mathcal{B})^{\op} \simeq
  \txt{RFib}(\mathcal{B})^{\op} \to \CatI$ that sends a right
  fibration $\mathcal{Y} \to \mathcal{B}$ to $\mathcal{P}(\mathcal{Y}
  \times_{\mathcal{B}} \mathcal{E})$.
\end{propn}
 To prove this, we need to identify the functor
 $\mathcal{P}(\mathcal{E})_{/p^{*}(\blank)}$
 with the functor $\mathcal{P}(\blank \times_{\mathcal{B}}
\mathcal{E})$ under the equivalence $\mathcal{P}(\mathcal{B}) \simeq
\txt{RFib}(\mathcal{B})$, for which we use the following observation:
\begin{propn}\label{propn:rfibeq}
  Suppose $p \colon \mathcal{K} \to \mathcal{C}$ is a right fibration
  of \icats{}. Then the functor \[p_{!} \colon \RFib(\mathcal{K}) \to
  \RFib(\mathcal{C})_{/p}\] given by composition with $p$ is an
  equivalence. Moreover, this equivalence is natural in $p \in
  \RFib(\mathcal{C})$ (with respect to composition with maps $f \colon
  \mathcal{K} \to \mathcal{L}$ over $\mathcal{C}$, and also with respect to pullbacks
  along such maps).
\end{propn}
\begin{proof}
  The functor $p_{!}$ is described by the left Quillen functor
  $p_{!} \colon \Set_{\Delta/\mathcal{K}} \to (\Set_{\Delta/\mathcal{C}})_{/p}$
 given by composition with $p$, where $\Set_{\Delta/\mathcal{K}}$ is equipped with the contravariant model
 structure of \cite{HTT}*{\S 2.1.4} and
 $(\Set_{\Delta/\mathcal{C}})_{/p}$ with the model structure induced
 from the contravariant model structure on
 $\Set_{\Delta/\mathcal{C}}$. It therefore suffices to show that $p_{!}$ is a
 left Quillen equivalence. The functor $p_{!}$ is obviously an equivalence of
 underlying categories, and we claim it is actually an equivalence of
 model categories. The cofibrations clearly correspond under $p_{!}$, being the
 monomorphisms of underlying simplicial sets in both cases, so by
 \cite{JoyalUABNotes}*{Proposition E.1.10} it suffices to show the fibrant objects are the
 same. In $\Set_{\Delta/\mathcal{K}}$ these are the right fibrations
 $\mathcal{X} \to \mathcal{K}$ by \cite{HTT}*{Corollary 2.2.3.12},
 while in $(\Set_{\Delta/\mathcal{C}})_{/p}$ they are the diagrams
 \opctriangle{\mathcal{X}}{\mathcal{K}}{\mathcal{C}}{f}{}{p}
 where $f$ is a fibration in $\Set_{\Delta/\mathcal{C}}$. But as $p$
 is a right fibration, this is equivalent to $f$ being a right
 fibration by \cite{HTT}*{Corollary 2.2.3.14}.
\end{proof}

\begin{cor}\label{cor:FuntoSeq}
  Suppose $p \colon \mathcal{K} \to \mathcal{C}$ is a right fibration,
  corresponding to a functor $F \colon \mathcal{C}^{\op} \to
  \mathcal{S}$. Then the functor $p_{!} \colon
  \mathcal{P}(\mathcal{K}) \to \mathcal{P}(\mathcal{C})_{/F}$ given by
  left Kan extension along $p^{\op}$ is an equivalence, natural in $p
  \in \RFib(\mathcal{C})$ (with respect to left Kan extensions along
  maps $f \colon \mathcal{K} \to \mathcal{L}$ over $\mathcal{C}$ and
  composition with the associated natural transformation, as well as
  with respect to composition with $f$ and pullback along the natural
  transformation).
\end{cor}
\begin{proof}
  This follows from combining Proposition~\ref{propn:rfibeq} with the
  naturality of the straightening equivalence between right fibrations
  and functors, which can be proved by the same argument as in the
  proof of Corollary~\ref{cor:unsteqnat}.
\end{proof}

\begin{proof}[Proof of Proposition~\ref{propn:PEBCart}]
  This follows from combining Lemma~\ref{lem:PEBident} with
  Corollary~\ref{cor:FuntoSeq}, since under the equivalence between
  presheaves and right fibrations the functor $p^{*} \colon
  \mathcal{P}(\mathcal{B}) \to \mathcal{P}(\mathcal{E})$ corresponds
  to pullback along $p$.
\end{proof}

\begin{cor}\label{cor:PEBfun}
  \ 
  \begin{enumerate}[(i)]
  \item Let $p \colon \mathcal{E} \to \mathcal{B}$ be a Cartesian
    fibration classified by a functor $F \colon \mathcal{B}^{\op} \to
    \CatI$, and write $j$ for the inclusion $\mathcal{B}
    \hookrightarrow \mathcal{E}_{\mathcal{B}}^{\triangleleft}$. Then
    the functor $j^{*} \colon
    \mathcal{P}(\mathcal{E}_{\mathcal{B}}^{\triangleleft}) \to
    \mathcal{P}(\mathcal{B})$ is a Cartesian fibration classified by
    the functor $\txt{RFib}(\mathcal{B})^{\op} \to \CatI$ that sends
    $\mathcal{Y} \to \mathcal{B}$ to
    $\Fun_{\mathcal{B}^{\op}}(\mathcal{Y}^{\op},
    \Phi^{F}_{\mathcal{S}})$, where $\Phi^{F}_{\mathcal{S}} \to
    \mathcal{B}^{\op}$ is the Cartesian fibration classified by the
    functor $\mathcal{P} \circ F \colon \mathcal{B} \to \CatI$.
  \item Let $p \colon \mathcal{E} \to \mathcal{B}$ be a coCartesian
    fibration classified by a functor $F \colon \mathcal{B} \to
    \CatI$, and write $j$ for the inclusion $\mathcal{B}
    \hookrightarrow \mathcal{E}_{\mathcal{B}}^{\triangleleft}$. Then
    the functor $j^{*} \colon
    \mathcal{P}(\mathcal{E}_{\mathcal{B}}^{\triangleleft}) \to
    \mathcal{P}(\mathcal{B})$ is a Cartesian fibration classified by 
    the functor $\txt{RFib}(\mathcal{B})^{\op} \to \CatI$ that sends
    $\mathcal{Y} \to \mathcal{B}$ to
    $\Fun_{\mathcal{B}^{\op}}(\mathcal{Y}^{\op},
    \widetilde{\Phi}^{F}_{\mathcal{S}})$, where $\widetilde{\Phi}^{F}_{\mathcal{S}} \to
    \mathcal{B}^{\op}$ is the coCartesian fibration classified by the
    functor $\mathcal{P} \circ F \colon \mathcal{B}^{\op} \to \CatI$.
  \end{enumerate}
\end{cor}
\begin{proof}
  Combine Proposition~\ref{propn:PEBCart} with
  Proposition~\ref{propn:UnFunrightadj} and its dual.
\end{proof}

Our next observation lets us identify several interesting functors
with full subfunctors of the functor
$\Fun_{\mathcal{B}^{\op}}((\blank)^{\op},
\widetilde{\Phi}^{F}_{\mathcal{S}})$, which will allow us to identify
the corresponding Cartesian fibrations with full subcategories of
$\mathcal{P}(\mathcal{E}_{\mathcal{B}}^{\triangleleft})$.

\begin{lemma}\label{lem:limextident}
  Suppose $F \colon \mathcal{B}^{\op} \to \CatI$ is a functor of
  \icats{} corresponding to the Cartesian fibration
  $p \colon \mathcal{E} \to \mathcal{B}$ and the coCartesian fibration
  $q \colon \mathcal{F} \to \mathcal{B}^{\op}$. Let $\widehat{F}
  \colon \mathcal{P}(\mathcal{B})^{\op} \to \CatI$ be the unique
  limit-preserving functor extending $F$. Then:
  \begin{enumerate}[(i)]
  \item The functor $\Fun^{\txt{cart}}_{\mathcal{B}}(\mathcal{B}_{/\blank},
    \mathcal{E}) \colon \mathcal{B}^{\op} \to \CatI$ is equivalent to $F$.
  \item The functor $\Fun^{\txt{cart}}_{\mathcal{B}}(\blank,
    \mathcal{E}) \colon \txt{RFib}(\mathcal{B}) \to \CatI$ corresponds
    to $\widehat{F}$ under the equivalence $\txt{RFib}(\mathcal{B})
    \simeq \mathcal{P}(\mathcal{B})$.
  \item The functor $\Fun^{\txt{cocart}}_{\mathcal{B}^{\op}}((\mathcal{B}^{\op})_{\blank/},
    \mathcal{F}) \colon \mathcal{B}^{\op} \to \CatI$ is equivalent to $F$.
  \item The functor $\Fun^{\txt{cocart}}_{\mathcal{B}^{\op}}((\blank)^{\op},
    \mathcal{F}) \colon \txt{RFib}(\mathcal{B}) \to \CatI$ corresponds
    to $\widehat{F}$ under the equivalence $\txt{RFib}(\mathcal{B})
    \simeq \mathcal{P}(\mathcal{B})$.
  \end{enumerate}
\end{lemma}
\begin{proof}
  We will prove (i) and (ii); the proofs of (iii) and (iv) are
  essentially the same. To prove (i), observe that the straightening
  equivalence between Cartesian fibrations and functors to $\CatI$
  gives us a natural equivalence
  \[
  \begin{split}
\Map_{\CatI}(\mathcal{C},
  \Fun^{\txt{cart}}_{\mathcal{B}}(\mathcal{B}_{/\blank}, \mathcal{E}))
  & \simeq \Map^{\txt{cart}}_{\mathcal{B}}(\mathcal{C} \times
  \mathcal{B}_{/\blank}, \mathcal{E}) \\ & \simeq
  \Map_{\Fun(\mathcal{B}^{\op}, \CatI)}(\mathcal{C} \times
  y_{\mathcal{B}}(\blank), F) \\ & \simeq \Map_{\Fun(\mathcal{B}^{\op},
    \CatI)}(y_{\mathcal{B}}(\blank), F^{\mathcal{C}}).    
  \end{split}
\]
  Now the Yoneda Lemma implies that this is naturally equivalent to
  $\iota F^{\mathcal{C}}(\blank) \simeq \Map_{\CatI}(\mathcal{C},
  F(\blank))$, and so we must have
  $\Fun^{\txt{cart}}_{\mathcal{B}}(\mathcal{B}_{/\blank}, \mathcal{E})
  \simeq F$. This proves (i).

  To prove (ii), we first observe that the functor
  $\Fun^{\txt{cart}}_{\mathcal{B}}(\blank, \mathcal{E})$ preserves
  limits, since for any \icat{} $\mathcal{C}$ we have
  \[ \Map_{\CatI}(\mathcal{C}, \Fun^{\txt{cart}}_{\mathcal{B}}(\blank,
  \mathcal{E})) \simeq \Map^{\txt{cart}}_{\mathcal{B}}((\mathcal{C}
  \times \mathcal{B}) \times_{\mathcal{B}} (\blank), \mathcal{E})\]
  and the Cartesian product in
  $\Cat_{\infty/\mathcal{B}}^{\txt{cart}}$ preserves colimits in each
  variable. Moreover, it follows from (i) that this functor extends
  $F$, since the right fibration $\mathcal{B}_{/b} \to \mathcal{B}$
  corresponds to the presheaf $y_{\mathcal{B}}(b)$ under the
  equivalence between $\txt{RFib}(\mathcal{B})$ and
  $\mathcal{P}(\mathcal{B})$.
\end{proof}

\begin{defn}
  Suppose $F \colon \mathcal{B} \to \CatI$ is a functor. Then we write
  $\mathcal{P}F \colon \mathcal{B}^{\op} \to \CatI$ for the composite of
  $F^{\op}$ with $\mathcal{P} \colon \CatI^{\op} \to \LCatI$, and let
  $\widehat{\mathcal{P}F} \colon \mathcal{P}(\mathcal{B})^{\op} \to
  \LCatI$ be the unique limit-preserving functor extending $\mathcal{P}F$.
\end{defn}

\begin{propn}\label{propn:hatPEBident}
  Let $F \colon \mathcal{B} \to \CatI$ be a functor, with $p \colon
  \mathcal{E} \to \mathcal{B}$ an associated coCartesian fibration. We
  define
  $\mathcal{P}^{\txt{cocart}}(\mathcal{E}_{\mathcal{B}}^{\triangleleft})$
  to be the full subcategory of
  $\mathcal{P}(\mathcal{E}_{\mathcal{B}}^{\triangleleft})$ spanned by
  those presheaves $\phi \colon
  (\mathcal{E}_{\mathcal{B}}^{\triangleleft})^{\op} \to \mathcal{S}$
  such that for every coCartesian morphism $\bar{\alpha} \colon e \to
  \alpha_{!}e$ in $\mathcal{E}$ over $\alpha \colon b \to b'$ in
  $\mathcal{B}$, the commutative square
  \csquare{\phi(\alpha_{!}e, 1)}{\phi(e, 1)}{\phi(b', 0)}{\phi(b,
    0)}{\phi(\bar{\alpha}, 1)}{}{}{\phi(\alpha, 0)}
  is a pullback square. Then the restricted projection
  $\mathcal{P}^{\txt{cocart}}(\mathcal{E}_{\mathcal{B}}^{\triangleleft})
  \to \mathcal{P}(\mathcal{B})$ is a Cartesian fibration classified by 
  the functor $\widehat{\mathcal{P}F}$.
\end{propn}
\begin{proof}
  Combining Lemma~\ref{lem:limextident} with (the dual of)
  Proposition~\ref{propn:UnFunrightadj}, we may identify
  $\widehat{\mathcal{P}F}$ with the functor
  $\Fun^{\txt{cocart}}_{\mathcal{B}^{\op}}((\blank)^{\op},
  \widetilde{\Phi}^{F}_{\mathcal{S}})$.  This is a natural full
  subcategory of $\Fun_{\mathcal{B}^{\op}}((\blank)^{\op},
  \widetilde{\Phi}^{F}_{\mathcal{S}})$, the functor classified by the
  Cartesian fibration $\mathcal{P}(\mathcal{E}_{\mathcal{B}}^{\triangleleft})
  \to \mathcal{P}(\mathcal{B})$ by Proposition~\ref{propn:PEBCart}. It
  follows that $\widehat{\mathcal{P}F}$ is classified by the
  projection to $\mathcal{P}(\mathcal{B})$ of the full subcategory of
  $\mathcal{P}(\mathcal{E}_{\mathcal{B}}^{\triangleleft})$ spanned by
  those presheaves that correspond to objects of $\Fun^{\txt{cocart}}_{\mathcal{B}^{\op}}((\blank)^{\op},
  \widetilde{\Phi}^{F}_{\mathcal{S}})$ under the identification of the
  fibres with $\Fun_{\mathcal{B}^{\op}}((\blank)^{\op},
  \widetilde{\Phi}^{F}_{\mathcal{S}})$.

  By \cite{HTT}*{Corollary 3.2.2.13}, the coCartesian edges of
  $\widetilde{\Phi}^{F}_{\mathcal{S}}$ are those that correspond to
  functors $\Delta^{1} \times_{\mathcal{B}^{\op}} \mathcal{E}^{\op}
  \to \mathcal{S}$ that take Cartesian edges of $\mathcal{E}^{\op}$
  to equivalences in $\mathcal{S}$. Thus if $\mathcal{F} \to
  \mathcal{B}^{\op}$ is a coCartesian fibration, a functor
  $\mathcal{F} \to \widetilde{\Phi}^{F}_{\mathcal{S}}$ over
  $\mathcal{B}^{\op}$ preserves coCartesian morphisms precisely if the
  classifying functor $\mathcal{F} \times_{\mathcal{B}^{\op}}
  \mathcal{E}^{\op} \to \mathcal{S}$ takes morphisms of the form
  $(\phi,\epsilon)$ with $\phi$ coCartesian in $\mathcal{F}$ and
  $\epsilon$ Cartesian in $\mathcal{E}^{\op}$ to equivalences in
  $\mathcal{S}$.

  If $\mathcal{Y} \to \mathcal{B}$ is a right fibration, this means
  that $\Fun^{\txt{cocart}}_{\mathcal{B}^{\op}}(\mathcal{Y}^{\op},
  \widetilde{\Phi}^{F}_{\mathcal{S}})$ corresponds to the full
  subcategory of $\mathcal{P}(\mathcal{Y} \times_{\mathcal{B}}
  \mathcal{E})$ spanned by the presheaves $(\mathcal{Y}
  \times_{\mathcal{B}} \mathcal{E})^{\op} \to \mathcal{S}$ that take
  morphisms of the form $(\eta, \epsilon)$ with $\epsilon$ coCartesian
  in $\mathcal{E}$ to equivalences in $\mathcal{S}$. 

  Suppose $\phi \colon \mathcal{B}^{\op} \to \mathcal{S}$ is the
  presheaf classified by $\mathcal{Y} \to \mathcal{B}$. Then unwinding
  the equivalence
  \[ \mathcal{P}(\mathcal{E}^{\triangleleft}_{\mathcal{B}})_{\phi}
  \simeq \mathcal{P}(\mathcal{E})_{/p^{*}\phi} \simeq
  \txt{RFib}(\mathcal{E})_{/p^{*}\mathcal{Y}} \simeq
  \txt{RFib}(\mathcal{E} \times_{\mathcal{B}}\mathcal{Y}) \simeq
  \mathcal{P}(\mathcal{E} \times_{\mathcal{B}}\mathcal{Y}),\]
  we see that the presheaf $\widetilde{\Psi}$ on $\mathcal{E} \times_{\mathcal{B}}
  \mathcal{Y}$ classified by $\Psi \colon \mathcal{E}^{\op} \to
  \mathcal{S}$ over $p^{*}\phi$ assigns to $(e, y) \in \mathcal{E}
  \times_{\mathcal{B}} \mathcal{Y}$ the fibre $\Psi(e)_{y}$ of the map $\Psi(e) \to
  \phi(pe)$ at $y$. Thus $\widetilde{\Psi}(\bar{\alpha}, \eta)$ is an
  equivalence for every coCartesian morphism $\bar{\alpha} \colon e
  \to \alpha_{!}e$ in
  $\mathcal{E}$ \IFF{} for every $y \in \phi(p(\alpha_{!}e))$ the map
  on fibres $\Psi(\alpha_{!}e)_{y} \to \Psi(e)_{\phi(\alpha)(e)}$ is
  an equivalence. This is equivalent to the commutative square
  \nolabelcsquare{\Psi(\alpha_{!}e)}{\Psi(e)}{\phi(p\alpha_{!}e)}{\phi(pe)}
  being Cartesian. Thus $\Fun^{\txt{cocart}}_{\mathcal{B}^{\op}}(\mathcal{Y}^{\op},
  \widetilde{\Phi}^{F}_{\mathcal{S}})$ corresponds to the full
  subcategory
  $\mathcal{P}^{\txt{cocart}}(\mathcal{E}^{\triangleleft}_{\mathcal{B}})_{\phi}$
  of $\mathcal{P}(\mathcal{E}^{\triangleleft}_{\mathcal{B}})_{\phi}$,
  which completes the proof.
\end{proof}

\begin{defn}
  Let $\mathcal{K}$ be a collection of small simplicial sets. We write
  $\CatI^{\mathcal{K}}$ for the subcategory of $\CatI$ with objects
  the small \icats{} that admit $K$-indexed colimits for all $K \in
  \mathcal{K}$ and morphisms the functors that preserve these. Given
  $\mathcal{C} \in \CatI^{\mathcal{K}}$ we let
  $\mathcal{P}_{\mathcal{K}}(\mathcal{C})$ denote the full subcategory
  of $\mathcal{P}(\mathcal{C})$ spanned by those presheaves
  $\mathcal{C}^{\op} \to \mathcal{S}$ that take $K$-indexed colimits
  in $\mathcal{C}$ to limits for all $K \in \mathcal{K}$. This defines
  a functor $\mathcal{P}_{\mathcal{K}} \colon
  (\CatI^{\mathcal{K}})^{\op} \to \LCatI$.
\end{defn}

\begin{ex}
  Let $\mathcal{K}(\kappa)$ be the collection of all $\kappa$-small
  simplicial sets. In this case we write $\CatI^{\kappa}$ for
  $\CatI^{\mathcal{K}(\kappa)}$ and $\mathcal{P}_{\kappa}$ for $\mathcal{P}_{\mathcal{K}(\kappa)}$. For $\mathcal{C} \in
  \CatI^{\kappa}$ the \icat{}
  $\mathcal{P}_{\kappa}(\mathcal{C})$ is equivalent to
  $\Ind_{\kappa}\mathcal{C}$ by \cite{HTT}*{Corollary 5.3.5.4}.
\end{ex}

\begin{defn}
  Suppose $\mathcal{K}$ is a collection of small simplicial sets and
  $F \colon \mathcal{B} \to \CatI^{\mathcal{K}}$ is a functor of \icats{}. Then we write
  $\mathcal{P}_{\mathcal{K}}F \colon \mathcal{B}^{\op} \to \LCatI$ for the composite of
  $F^{\op}$ with $\mathcal{P}_{\mathcal{K}} \colon (\CatI^{\mathcal{K}})^{\op} \to \LCatI$, and let
  $\widehat{\mathcal{P}_{\mathcal{K}}F} \colon \mathcal{P}(\mathcal{B})^{\op} \to
  \LCatI$ be the unique limit-preserving functor extending
  $\mathcal{P}_{\mathcal{K}}F$. 
\end{defn}

\begin{propn}
  Let $\mathcal{K}$ be a collection of small simplicial sets and let
  $F \colon \mathcal{B} \to \CatI^{\mathcal{K}}$ be a functor, with $p
  \colon \mathcal{E} \to \mathcal{B}$ an associated coCartesian
  fibration. We define
  $\mathcal{P}^{\txt{cocart}}_{\mathcal{K}}(\mathcal{E}_{\mathcal{B}}^{\triangleleft})$
  to be the full subcategory of
  $\mathcal{P}^{\txt{cocart}}(\mathcal{E}_{\mathcal{B}}^{\triangleleft})$
  spanned by those presheaves $\phi \colon
  (\mathcal{E}_{\mathcal{B}}^{\triangleleft})^{\op} \to \mathcal{S}$
  such that for every colimit diagram $\bar{q} \colon
  K^{\triangleright} \to \mathcal{E}_{b}$ (for any $b \in
  \mathcal{B}$), the composite
  \[ (K^{\triangleleft\triangleright})^{\op} \to
  (\mathcal{E}_{b}^{\triangleleft})^{\op} \to
  (\mathcal{E}^{\triangleleft}_{\mathcal{B}})^{\op} \to \mathcal{S}\]
  is a limit diagram. Then the restricted projection
  $\mathcal{P}^{\txt{cocart}}_{\mathcal{K}}(\mathcal{E}_{\mathcal{B}}^{\triangleleft})
  \to \mathcal{P}(\mathcal{B})$ is a Cartesian fibration associated to
  the functor $\widehat{\mathcal{P}_{\mathcal{K}}F}$.
\end{propn}
\begin{proof}
  Since $\mathcal{P}_{\mathcal{K}}F(b)$ is a natural full subcategory
  of $\mathcal{P}F(b)$, we may identify the coCartesian fibration
classified by $\mathcal{P}_{\mathcal{K}}F$ with the projection to
  $\mathcal{B}^{\op}$ of a full subcategory
  $\widetilde{\Phi}^{F}_{\mathcal{S}}(\mathcal{K})$ of
  $\widetilde{\Phi}^{F}_{\mathcal{S}}$. By Lemma~\ref{lem:limextident}
  we may then identify the functor
  $\widehat{\mathcal{P}_{\mathcal{K}}F}$ with
  $\Fun^{\txt{cocart}}_{\mathcal{B}^{\op}}((\blank)^{\op},
  \widetilde{\Phi}^{F}_{\mathcal{S}}(\mathcal{K}))$, which is a
  natural full subcategory of $\Fun^{\txt{cocart}}_{\mathcal{B}^{\op}}((\blank)^{\op},
  \widetilde{\Phi}^{F}_{\mathcal{S}})$. By
  Proposition~\ref{propn:hatPEBident} we may therefore identify the
  Cartesian fibration classified by
  $\widehat{\mathcal{P}_{\mathcal{K}}F}$ with the projection to
  $\mathcal{P}(\mathcal{B})$ of a full subcategory of
  $\mathcal{P}^{\txt{cocart}}(\mathcal{E}_{\mathcal{B}}^{\triangleleft})$.
  
  It thus remains to identify those presheaves on
  $\mathcal{E}_{\mathcal{B}}^{\triangleleft}$ that correspond to
  objects of the \icat{}
  $\Fun^{\txt{cocart}}_{\mathcal{B}^{\op}}((\blank)^{\op},
  \widetilde{\Phi}^{F}_{\mathcal{S}}(\mathcal{K}))$ under the
  identification of the fibres with
  $\Fun_{\mathcal{B}^{\op}}((\blank)^{\op},
  \widetilde{\Phi}^{F}_{\mathcal{S}})$. If $\mathcal{Y} \to
  \mathcal{B}$ is a right fibration, it is clear that
  $\Fun_{\mathcal{B}^{\op}}(\mathcal{Y}^{\op},
  \widetilde{\Phi}^{F}_{\mathcal{S}}(\mathcal{K}))$ corresponds to the
  full subcategory of $\mathcal{P}(\mathcal{Y} \times_{\mathcal{B}}
  \mathcal{E})$ spanned by the presheaves $(\mathcal{Y}
  \times_{\mathcal{B}} \mathcal{E})^{\op} \to \mathcal{S}$ such that
  for every $y \in \mathcal{Y}$ over $b \in \mathcal{B}$, the
  restriction $(\{y\} \times_{\mathcal{B}} \mathcal{E})^{\op} \simeq
  \mathcal{E}_{b}^{\op} \to \mathcal{S}$ preserves $K$-indexed limits
  for all $K \in \mathcal{K}$.

  Suppose $\phi \colon \mathcal{B}^{\op} \to \mathcal{S}$ is the
  presheaf classified by $\mathcal{Y} \to \mathcal{B}$. Then unwinding
  the equivalence
  \[ \mathcal{P}(\mathcal{E}^{\triangleleft}_{\mathcal{B}})_{\phi}
  \simeq \mathcal{P}(\mathcal{E})_{/p^{*}\phi} \simeq
  \txt{RFib}(\mathcal{E})_{/p^{*}\mathcal{Y}} \simeq
  \txt{RFib}(\mathcal{E} \times_{\mathcal{B}}\mathcal{Y}) \simeq
  \mathcal{P}(\mathcal{E} \times_{\mathcal{B}}\mathcal{Y}),\]
  we see that these presheaves on $\mathcal{E} \times_{\mathcal{B}}
  \mathcal{Y}$ correspond to presheaves $\Psi \colon \mathcal{E}^{\op}
  \to \mathcal{S}$ over $p^{*}\phi$ such that for every $b \in
  \mathcal{B}$, the restriction
  \[ \mathcal{E}_{b}^{\op} \to \mathcal{S}_{/\phi(b)} \]
  has the property that for every $y \in \phi(b)$, the composite with
  the map
  $\mathcal{S}_{/\phi(b)} \to \mathcal{S}$ given by pullback along
  $\{y\} \to \phi(b)$ takes $K$-indexed colimits to limits for all $K
  \in \mathcal{K}$.

  Now recall that the \icat{} $\mathcal{S}_{/\phi(b)}$ is equivalent to $\Fun(\phi(b),
  \mathcal{S})$, with pullback to $\{y\}$ corresponding to evaluation
  at $y$, and that limits in functor categories are computed
  pointwise. Thus we may identify our full subcategory with that of
  presheaves $\Psi$ such that for every $b \in \mathcal{B}$, the
  restriction
  \[ \mathcal{E}_{b}^{\op} \to \mathcal{S}_{/\phi(b)} \]
  takes $K$-indexed colimits to limits in $\mathcal{S}_{/\phi(b)}$.

  For any \icat{} $\mathcal{C}$ and $x \in \mathcal{C}$, a diagram
  $K^{\triangleleft} \to \mathcal{C}_{/x}$ is a limit diagram \IFF{}
  the associated diagram $K^{\triangleleft\triangleright} \to
  \mathcal{C}$ is a limit diagram. Therefore, the full subcategory of
  $\mathcal{P}^{\txt{cocart}}(\mathcal{E}_{\mathcal{B}}^{\triangleleft})$
  we have identified is precisely
  $\mathcal{P}^{\txt{cocart}}_{\mathcal{K}}(\mathcal{E}_{\mathcal{B}}^{\triangleleft})$,
  which completes the proof.
\end{proof}

\begin{cor}\label{cor:PEBIndkappaident}
  Suppose $F \colon \mathcal{B} \to \CatI^{\kappa}$ is a functor, with $p
  \colon \mathcal{E} \to \mathcal{B}$ an associated coCartesian
  fibration. Let $\widehat{\Ind}_{\kappa}F \colon
  \mathcal{P}(\mathcal{B})^{\op} \to \LCatI$ be the unique
  limit-preserving functor extending $\Ind_{\kappa} \circ F^{\op}
 \colon \mathcal{B}^{\op} \to \LCatI$. Then the restricted projection
  $\mathcal{P}^{\txt{cocart}}_{\kappa}(\mathcal{E}_{\mathcal{B}}^{\triangleleft})
  := \mathcal{P}^{\txt{cocart}}_{\mathcal{K}(\kappa)}(\mathcal{E}_{\mathcal{B}}^{\triangleleft})
  \to \mathcal{P}(\mathcal{B})$ is a Cartesian fibration classified by 
  the functor $\widehat{\Ind}_{\kappa}F$.
  
\end{cor}

\section{Presentable Fibrations are Presentable}\label{sec:pres}
In ordinary category theory, an \emph{accessible fibration} is a
Grothendieck fibration $p \colon \mathbf{E} \to \mathbf{C}$ such that
$\mathbf{C}$ is an accessible category, the corresponding functor $F
\colon \mathbf{C}^{\op} \to \widehat{\Cat}$ factors through the
category of accessible categories and accessible functors, and $F$
preserves $\kappa$-filtered limits for $\kappa$ sufficiently large.

In \cite{MakkaiPare}, Makkai and Paré prove that if $p$ is an
accessible fibration, then its source $\mathbf{E}$ is also an
accessible category, and $p$ is an accessible functor. The goal of
this section is to prove an \icatl{} variant of this result. As it
makes the proof much clearer we will, however, restrict ourselves to
considering only \emph{presentable fibrations} of \icats{}, defined as
follows:
\begin{defn}
  A \defterm{presentable fibration} is a Cartesian fibration $p \colon
  \mathcal{E} \to \mathcal{B}$ such that $\mathcal{B}$ is a
  presentable \icat{}, the corresponding functor $F \colon \mathcal{B}^{\op} \to
  \LCatI$ factors through the \icat{} $\PrR$
  of presentable \icats{} and right adjoints, and $F$ preserves
  $\kappa$-filtered limits for $\kappa$ sufficiently large.
\end{defn}
\begin{remark}
  Suppose $p \colon \mathcal{E} \to \mathcal{B}$ is a presentable
  fibration. Since the morphisms of $\mathcal{B}$ are all mapped to
  right adjoints under the associated functor, it follows that $p$ is
  also a coCartesian fibration.
\end{remark}
The goal of this section is then to prove the following:
\begin{thm}\label{thm:PresFib}
  Let $p \colon \mathcal{E} \to \mathcal{B}$ be a presentable
  fibration. Then $\mathcal{E}$ is a presentable \icat{}.
\end{thm}

As in Makkai and Paré's proof of \cite{MakkaiPare}*{Theorem 5.3.4}, we
will prove this by explicitly describing the total space of the
presentable fibration classified by a special class of functors as an
accessible localization of a presheaf \icat{}. To state this result we
first recall some notation from \cite{HTT}*{\S 5.5.7}:
\begin{defn}
  Suppose $\kappa$ is a regular cardinal. As before, let $\CatI^{\kappa}$ be
  the category of small \icats{} that have all $\kappa$-small
  colimits, and functors that preserve these. Then $\Ind_{\kappa}$ gives a functor from $\CatI^{\kappa,\op}$
  to the \icat{} $\PrRk$ of $\kappa$-presentable
  \icats{} and limit-preserving functors that preserve
  $\kappa$-filtered colimits.  Using the equivalence $\PrL \simeq
  (\PrR)^{\op}$ we may equivalently regard this as a functor
  $\Ind_{\kappa}^{\vee} \colon \CatI^{\kappa} \to \PrLk$,
  where $\PrLk$ is the \icat{} of $\kappa$-presentable
  \icats{} and colimit-preserving functors that preserve
  $\kappa$-compact objects.
\end{defn}
The key step in the proof of Theorem~\ref{thm:PresFib} can then be
stated as follows:
\begin{propn}\label{propn:PS}
  Suppose $F \colon \mathcal{B} \to \CatI^{\kappa}$ is a functor of
  \icats{} with associated coCartesian fibration
  $p \colon \mathcal{E} \to \mathcal{B}$. Let $q \colon \widehat{\mathcal{E}}
  \to \mathcal{P}(\mathcal{B})$ be a Cartesian fibration classified by the unique limit-preserving functor $\widehat{\Ind}_{\kappa}F
  \colon \mathcal{P}(\mathcal{B})^{\op} \to \PrRk$
  extending $\Ind_{\kappa} \circ F^{\op} \colon \mathcal{B}^{\op} \to
  \PrRk$. Then the \icat{} $\widehat{\mathcal{E}}$ is
  an accessible \icat{}, and $q$ is an accessible functor.
\end{propn}

We will prove Proposition~\ref{propn:PS} using Corollary~\ref{cor:PEBIndkappaident}
together with the following simple observation:
\begin{lemma}\label{lem:PresheafLimAcc}
  Suppose $\mathcal{C}$ is a small \icat{}, and let $S =
  \{\overline{p}_{\alpha} \colon K_{\alpha}^{\triangleright} \to
  \mathcal{C}\}$ be a small set of diagrams in $\mathcal{C}$. Then the
  full subcategory of $\mathcal{P}(\mathcal{C})$ spanned by presheaves
  that take the diagrams in $S$ to limit diagrams in $\mathcal{S}$
  is accessible, and the inclusion of this into
  $\mathcal{P}(\mathcal{C})$ is an accessible functor.
\end{lemma}
\begin{proof}
  Let
  $y_{\mathcal{C}} \colon \mathcal{C} \to \mathcal{P}(\mathcal{C})$
  denote the Yoneda embedding. A presheaf
  $F \colon \mathcal{C}^{\op} \to \mathcal{S}$ takes
  $\overline{p}_{\alpha}^{\op}$ to a limit diagram \IFF{} it is local
  with respect to the map of presheaves
  \[\colim (y_{\mathcal{C}} \circ \overline{p}|_{K_{\alpha}}) \to
  y_{\mathcal{C}}(\infty),\]
  where $\infty$ denotes the cone point. Thus if $S'$ is the set of
  these morphisms for $\overline{p}_{\alpha} \in S$, the subcategory
  in question is precisely the full subcategory of $S'$-local
  objects. (This observation can also be found e.g.\ in the proof of
  \cite{HTT}*{Proposition 5.3.6.2}.) Since $S$, and hence $S'$, is by
  assumption a small set, it follows that this subcategory is an
  accessible localization of $\mathcal{P}(\mathcal{C})$. In
  particular, it is itself accessible and the inclusion into
  $\mathcal{P}(\mathcal{C})$ is an accessible functor.
\end{proof}

\begin{proof}[Proof of Proposition~\ref{propn:PS}]
  By Proposition~\ref{cor:PEBIndkappaident} the Cartesian fibration
  $\widehat{\mathcal{E}} \to \mathcal{P}(\mathcal{B})$ can be identified with
  the restriction to the full subcategory
  $\mathcal{P}^{\txt{cocart}}_{\kappa}(\mathcal{E}_{\mathcal{B}}^{\triangleleft})$
  of the functor $j^{*} \colon
  \mathcal{P}(\mathcal{E}_{\mathcal{B}}^{\triangleleft})
  \to \mathcal{P}(\mathcal{B})$ induced by composition with the
  inclusion $j \colon \mathcal{B} \hookrightarrow
  \mathcal{E}_{\mathcal{B}}^{\triangleleft}$.

  The \icat{}
  $\mathcal{P}^{\txt{cocart}}_{\kappa}(\mathcal{E}_{\mathcal{B}}^{\triangleleft})$
  is by definition the full subcategory of
  $\mathcal{P}(\mathcal{E}_{\mathcal{B}}^{\triangleleft})$ spanned by
  presheaves that take two
  classes of diagrams to limit diagrams in $\mathcal{S}$ --- one
  indexed by coCartesian morphisms in $\mathcal{E}$, which form a set,
  and one indexed by $\kappa$-small colimit diagrams in the fibres of $p$;
  these do not form a set, but we can equivalently consider only
  pushout squares and coproducts indexed by $\kappa$-small sets, which
  \emph{do} form a set. It then follows from
  Lemma~\ref{lem:PresheafLimAcc} that 
  $\mathcal{P}^{\txt{cocart}}_{\kappa}(\mathcal{E}_{\mathcal{B}}^{\triangleleft})$
  is accessible and the inclusion
  $\mathcal{P}^{\txt{cocart}}_{\kappa}(\mathcal{E}_{\mathcal{B}}^{\triangleleft})
  \hookrightarrow
  \mathcal{P}(\mathcal{E}_{\mathcal{B}}^{\triangleleft})$ is an
  accessible functor. The functor $j^{*} \colon
  \mathcal{P}(\mathcal{E}_{\mathcal{B}}^{\triangleleft}) \to
  \mathcal{P}(\mathcal{B})$ preserves colimits, since these are
  computed pointwise, and so the composite
  $\mathcal{P}^{\txt{cocart}}_{\kappa}(\mathcal{E}_{\mathcal{B}}^{\triangleleft})
  \to \mathcal{P}(\mathcal{B})$ is also an accessible functor.
\end{proof}

To complete the proof of Theorem~\ref{thm:PresFib} we now just need
an easy Lemma:
\begin{lemma}\label{lem:coCartcolims}
  Suppose $\pi \colon \mathcal{E} \to \mathcal{B}$ is a coCartesian
  fibration such that both $\mathcal{B}$ and the fibres
  $\mathcal{E}_{b}$ for all $b \in \mathcal{B}$ admit small colimits,
  and the functors $f_{!} \colon \mathcal{E}_{b} \to \mathcal{E}_{b'}$
  preserve colimits for all morphisms $f \colon b \to b'$ in
  $\mathcal{B}$. Then $\mathcal{E}$ admits small colimits.
\end{lemma}
\begin{proof}
  The coCartesian fibration $\pi$ satisfies the conditions of \cite[Corollary
  4.3.1.11]{HTT} for all small simplicial sets $K$, and so in every
  diagram
  \liftcsquare{K}{\mathcal{E}}{K^{\triangleright}}{\mathcal{B}}{p}{}{\pi}{\bar{q}}{\bar{p}}
  there exists a lift $\bar{p}$ that is a $\pi$-colimit of $p$. Given
  a diagram $p \colon K \to \mathcal{E}$
  we can apply this with $\bar{q}$ a colimit of $\pi \circ p$ to get a
  colimit $\bar{p} \colon K^{\triangleright} \to
  \mathcal{E}$ of $p$.
\end{proof}

\begin{proof}[Proof of Theorem~\ref{thm:PresFib}]
  It follows from Lemma~\ref{lem:coCartcolims} that $\mathcal{E}$ has
  small colimits. It thus remains to prove that $\mathcal{E}$ is
  accessible and $p$ is an accessible functor. Let $F \colon
  \mathcal{B}^{\op}\to \LCatI$ be a functor corresponding to $p$. Choose
  a regular cardinal $\kappa$ so that $\mathcal{B}$ is $\kappa$-presentable
  and $F$ preserves $\kappa$-filtered limits. Since $\mathcal{B}$ is
  $\kappa$-presentable, $\mathcal{B} \simeq
  \Ind_{\kappa}\mathcal{B}^{\kappa}$ is the full subcategory of
  $\mathcal{P}(\mathcal{B}^{\kappa})$ spanned by the presheaves that
  preserve $\kappa$-small limits. Let $\widehat{F} \colon
  \mathcal{P}(\mathcal{B}^{\kappa})^{\op} \to \LCatI$ be the unique
  limit-preserving functor extending $F|_{\mathcal{B}^{\kappa,\op}}$;
  then $F$ is equivalent to the restriction of $\widehat{F}$ to
  $\Ind_{\kappa}\mathcal{B}^{\kappa}$. If $\widehat{p} \colon
  \widehat{\mathcal{E}} \to \mathcal{P}(\mathcal{B}^{\kappa})$ is a
  Cartesian fibration classified by $\widehat{F}$ we therefore have a
  pullback square
  \nolabelcsquare{\mathcal{E}}{\widehat{\mathcal{E}}}{\mathcal{B}}{\mathcal{P}(\mathcal{B}^{\kappa}),}
  where the bottom map preserves $\kappa$-filtered colimits, so by
  \cite{HTT}*{Proposition 5.4.6.6} it suffices to show that
  $\widehat{\mathcal{E}}$ is accessible and $\widehat{p}$ is an
  accessible functor.

  Since $\mathcal{B}^{\kappa}$ is a small \icat{}, we can choose a
  cardinal $\lambda$ such that $F|_{\mathcal{B}^{\kappa,\op}}$ factors
  through the \icat{} $\Pr^{\txt{R},\lambda}$ of $\lambda$-presentable
  \icats{} and right adjoints that preserve $\lambda$-filtered
  colimits. By \cite{HTT}*{Proposition 5.5.7.2} we can equivalently think
  of this, via the equivalence $\PrR \simeq (\PrL)^{\op}$, as a
  functor from $\mathcal{B}^{\kappa}$ to the \icat{}
  $\Pr^{\txt{L},\lambda}$ of $\lambda$-presentable \icats{} and
  functors that preserve colimits and $\lambda$-compact
  objects. Taking $\lambda$-compact objects defines a functor
  $(\blank)^{\lambda}\colon \Pr^{\txt{L},\lambda} \to
  \CatI^{\lambda}$. Then, defining $F_{0} \colon
  (\mathcal{B}^{\kappa}) \to \CatI$ to be
  $(F^{\op}|_{\mathcal{B}^{\kappa}})^{\lambda}$, we see that $F \simeq
  \Ind_{\lambda}F_{0}$, and so $\widehat{\mathcal{E}}$ is accessible
  and $\widehat{p}$ is an accessible functor by
  Proposition~\ref{propn:PS}.
\end{proof}

\appendix

\section{Pseudofunctors and the Naturality of Unstraightening}\label{sec:pseudo}
At several points in this paper we will need to know that the
unstraightening functors $\Fun(\mathcal{C}^{\op}, \CatI) \to \CatICc$
(and a number of similar constructions) are natural as we vary the
\icat{} $\mathcal{C}$. The obvious way to prove this is to consider
the naturality of the unstraightening $\txt{Un}_{S}^{+}\colon
\Fun(\mathfrak{C}(S), \sSet^{+}) \to (\sSet^{+})_{/S}$ as we vary the
simplicial set $S$. However, since pullbacks are only determined up to
canonical isomorphism, these functors are not natural ``on the nose'',
but only up to natural isomorphism --- i.e.\  they are only
\emph{pseudo-natural}. In the body of the paper we have swept such
issues under the rug, but in this appendix we indulge ourselves in a
bit of 2-category theory to prove that pseudo-naturality on the level
of model categories does indeed give naturality on the level of
\icats{}. We begin by reviewing Duskin's nerve of
bicategories~\cite{Duskin} and its basic properties. However, we will
only need to consider the case of \emph{strict} 2- and
(2,1)-categories:
\begin{defn}
  A \emph{strict 2-category} is a category enriched in $\Cat$, and a
  \emph{strict (2,1)-category} is a category enriched in $\Gpd$. We
  write $\Cat_{2}$ for the category of strict 2-categories and
  $\Cat_{(2,1)}$ for the category of strict (2,1)-categories.
\end{defn}

\begin{defn}\label{defn:oplax}
  Suppose $\mathbf{C}$ and $\mathbf{D}$ are strict 2-categories. A
  \emph{normal oplax functor} $F \colon \mathbf{C} \to \mathbf{D}$ consists
  of the following data:
  \begin{enumerate}[(a)]
  \item for each object $x \in \mathbf{C}$, an object $F(x) \in \mathbf{D}$,
  \item for each 1-morphism $f \colon x \to y$ in $\mathbf{C}$, a
    1-morphism $F(f) \colon F(x) \to F(y)$,
  \item for each 2-morphism $\phi \colon f \To g$ in
    $\mathbf{C}(x,y)$, a 2-morphism $F(\phi) \colon F(f) \To F(g)$ in
    $\mathbf{D}(F(x), F(y))$,
  \item for each pair of composable 1-morphisms $f \colon x \to y$, $g
    \colon y \to z$ in $\mathbf{C}$, a 2-morphism
    $\eta_{f,g} \colon F(g \circ f) \To F(g) \circ F(f)$,
  \end{enumerate}
  such that:
  \begin{enumerate}[(i)]
  \item for every object $x \in \mathbf{C}$, the 1-morphism
    $F(\id_{x}) = \id_{F(x)}$,
  \item for every 1-morphism $f \colon x \to y$ in $\mathbf{C}$, the
    2-morphism $F(\id_{f}) = \id_{F(f)}$, 
  \item for composable 2-morphisms $\phi \colon f \To g$, $\psi \colon
    g \To h$ in $\mathbf{C}(x,y)$, we have $F(\psi \circ \phi) =
    F(\psi) \circ F(\phi)$,
  \item for every morphism $f \colon x \to y$, the morphisms
    $\eta_{\id_{x},f}$ and $\eta_{f,\id_{y}} \colon F(f) \to F(f)$ are
    both $\id_{F(f)}$,
  \item if $\phi \colon f \To f'$ is a 2-morphism in
    $\mathbf{C}(x,y)$ and $\psi \colon g \To g'$ is a 2-morphism in
    $\mathbf{C}(y,z)$, then the diagram
    \csquare{F(g\circ f)}{F(g) \circ F(f)}{F(g'\circ f')}{F(g')
      \circ F(f')}{\eta_{f,g}}{F(\psi \circ \phi)}{F(\psi) \circ
      F(\phi)}{\eta_{f',g'}}
    commutes,
  \item for composable triples of 1-morphisms $f \colon x \to y$, $g
    \colon y \to z$, $h \colon z \to w$, the diagram
    \csquare{F(h \circ g\circ f)}{F(h \circ g) \circ F(f)}{F(h)
      \circ F(g \circ f)}{F(h) \circ F(g) \circ
      F(f))}{\eta_{f,hg}}{\eta_{gf,h}}{\eta_{g,h}\circ \id}{\id \circ
      \eta_{f,g}}
    commutes.
  \end{enumerate}
  We say a normal oplax functor $F$ from $\mathbf{C}$ to $\mathbf{D}$
  is a \emph{normal pseudofunctor} if the 2-morphisms $\eta_{f,g}$ are
  all isomorphisms. In particular, if the 2-category $\mathbf{C}$
  is a (2,1)-category, all normal oplax functors $\mathbf{C} \to
  \mathbf{D}$ are normal pseudofunctors.
\end{defn}

\begin{remark}
  In 2-category theory one typically considers the more general
  notions of (not necessarily normal) oplax functors and
  pseudofunctors, which do not satisfy $F(\id_{x}) = \id_{F(x)}$ but
  instead include the data of natural maps $F(\id_{x}) \to \id_{F(x)}$
  (which are isomorphisms for pseudofunctors). We only consider the
  normal versions because, as we will see below, these correspond to
  maps of simplicial sets between the \emph{nerves} of strict 2- and
  (2,1)-categories.
\end{remark}

Before we recall the definition of the nerve of a strict 2-category,
we first review the definition of nerves for ordinary categories and
simplicial categories:
\begin{defn}
  Let $\mathrm{N} \colon \Cat \to \sSet$ be the usual nerve of
  categories, i.e.\  if $\mathbf{C}$ is a category then
  $\mathrm{N}\mathbf{C}_{k}$ is the set $\Hom([k], \mathbf{C})$
  where $[k]$ is the category corresponding to the partially ordered
  set $\{0, \ldots, k\}$.
\end{defn}
\begin{remark}
  Since $\Cat$ has colimits, the functor $\mathrm{N}$ has a left
  adjoint $\mathrm{C} \colon \sSet \to \Cat$, which is the unique
  colimit-preserving functor such that $\mathrm{C}(\Delta^{n}) = [n]$.
\end{remark}

\begin{lemma}\label{lem:Cinneranodyne}
  The functor $\mathrm{C} \colon \sSet \to \Cat$ takes inner anodyne
  morphisms to isomorphisms.
\end{lemma}
\begin{proof}
  Let $\mathfrak{W}$ denote the class of monomorphisms of simplicial sets that
  are taken to isomorphisms by $\mathrm{C}$. To see that
  $\mathfrak{W}$ contains the inner anodyne morphisms we apply
  \cite{JoyalTierney}*{Lemma 3.5}, which says that $\mathfrak{W}$
  contains the inner anodyne maps if
  \begin{enumerate}[(i)]
  \item $\mathfrak{W}$ is weakly saturated, i.e.\ it contains the
    isomorphisms and is closed under composition, transfinite
    composition, cobase change, and codomain retracts,
  \item $\mathfrak{W}$ has the right cancellation property, i.e.\ it
    if $fg$ and $g$ are in $\mathfrak{W}$ then $f$ is in $\mathfrak{W}$,
  \item $\mathfrak{W}$ contains the inclusions
    $\txt{Sp}^{n} \hookrightarrow \Delta^{n}$, where $\txt{Sp}^{n}$
    denotes the $n$-spine, i.e.\ the simplicial set
    $\Delta^{\{0,1\}} \amalg_{\Delta^{\{1\}}} \cdots
    \amalg_{\Delta^{\{n-1\}}} \Delta^{\{n-1,n\}}$. 
  \end{enumerate}
  Here conditions (i) and (ii) follow immediately from the definition
  of $\mathfrak{W}$, as the functor $\mathrm{C}$ preserves
  colimits. It remains to prove (iii), i.e.\ to show that
  $\mathrm{C}(\txt{Sp}^{n}) \to \mathrm{C}(\Delta^{n})$ is an
  isomorphism. Since $\mathrm{C}$ preserves colimits, this is the map
  of categories
  \[ [1] \amalg_{[0]} \cdots \amalg_{[0]} [1] \to [n]. \]
  But the category $[n]$ is the free category on the graph
  with vertices $0, \ldots, n$ and edges $i \to (i+1)$, which obviously
  decomposes as a colimit in this way, and the free category functor
  on graphs preserves colimits.
\end{proof}

\begin{propn}\label{propn:Cprod} 
  The functor $\mathrm{C} \colon \sSet \to \Cat$ preserves products.
\end{propn}
\begin{proof}
  Since $\mathrm{C}$ preserves colimits and the Cartesian products in
  $\Cat$ and $\sSet$ both commute with colimits in each variable, it
  suffices to check that the natural map $\mathrm{C}(\Delta^{n}\times
  \Delta^{m}) \to \mathrm{C}(\Delta^{n}) \times
  \mathrm{C}(\Delta^{m})$ is an isomorphism for all $n, m$. Since
  products of inner anodyne maps are inner anodyne by
  \cite{HTT}*{Corollary 2.3.2.4}, the inclusion $\txt{Sp}^{n} \times
  \txt{Sp}^{m} \to \Delta^{n} \times \Delta^{m}$ is inner
  anodyne. Thus in the diagram \nolabelcsquare{\mathrm{C}(\txt{Sp}^{n}
    \times \txt{Sp}^{m})}{\mathrm{C}(\txt{Sp}^{n}) \times
    \mathrm{C}(\txt{Sp}^{m})}{\mathrm{C}(\Delta^{n} \times
    \Delta^{m})}{\mathrm{C}(\Delta^{n}) \times \mathrm{C}(\Delta^{m})}
  the vertical maps are isomorphisms by
  Lemma~\ref{lem:Cinneranodyne}. It hence suffices to prove that the
  upper horizontal map is an isomorphism. Since $\mathrm{C}$
  preserves colimits and the Cartesian products preserve colimits in
  each variable, in the commutative diagram
  \[
 \begin{tikzcd}[column sep=small]
   \mathrm{C}(\Delta^{1} \times \txt{Sp}^{m})
   \amalg_{\mathrm{C}(\Delta^{0} \times \txt{Sp}^{m})}
   \mathrm{C}(\txt{Sp}^{n} \times \txt{Sp}^{m}) \arrow{d} \arrow{r} & \mathrm{C}(\txt{Sp}^{n+1}
   \times \txt{Sp}^{m}) \arrow{d} \\
\mathrm{C}(\Delta^{1}) \times \mathrm{C}(\txt{Sp}^{m})
   \amalg_{\mathrm{C}(\Delta^{0}) \times \mathrm{C}(\txt{Sp}^{m})}
   \mathrm{C}(\txt{Sp}^{n}) \times \mathrm{C}(\txt{Sp}^{m}) \arrow{r}
   & \mathrm{C}(\txt{Sp}^{n+1}) \times
   \mathrm{C}(\txt{Sp}^{m})
 \end{tikzcd}
 \]
  \[
 \begin{tikzcd}
   \mathrm{C}(\Delta^{1} \times \txt{Sp}^{m})
   \underset{\mathrm{C}(\Delta^{0} \times \txt{Sp}^{m})}{\amalg}
   \mathrm{C}(\txt{Sp}^{n} \times \txt{Sp}^{m}) \arrow{d} \arrow{r} & \mathrm{C}(\txt{Sp}^{n+1}
   \times \txt{Sp}^{m}) \arrow{d} \\
\mathrm{C}(\Delta^{1}) \times \mathrm{C}(\txt{Sp}^{m})
   \underset{\mathrm{C}(\Delta^{0}) \times \mathrm{C}(\txt{Sp}^{m})}{\amalg}
   \mathrm{C}(\txt{Sp}^{n}) \times \mathrm{C}(\txt{Sp}^{m}) \arrow{r}
   & \mathrm{C}(\txt{Sp}^{n+1}) \times
   \mathrm{C}(\txt{Sp}^{m})
 \end{tikzcd}
 \]
  the vertical morphisms are isomorphisms. By inducting on $n$ and $m$
  this implies that the map in question is an isomorphism for all $n$ and $m$
  provided \[\mathrm{C}(\Delta^{n}
  \times \Delta^{m}) \to \mathrm{C}(\Delta^{n}) \times
  \mathrm{C}(\Delta^{m})\] is an isomorphism when $n$ and $m$ are both
  either $0$ or $1$. The cases where $n$ or $m$ is $0$ are trivial, so
  it only remains to show that $\mathrm{C}(\Delta^{1} \times \Delta^{1}) \to
  [1] \times [1]$ is an isomorphism. The simplicial set $\Delta^{1}
  \times \Delta^{1}$ is the pushout $\Delta^{2}
  \amalg_{\Delta^{\{0,2\}}} \Delta^{2}$, so this amounts to showing
  that the analogous functor $[2] \amalg_{[1]} [2] \to [1] \times [1]$
  is an isomorphism, or equivalently that for any category
  $\mathbf{C}$, the square \nolabelcsquare{\Hom([1] \times [1],
    \mathbf{C})}{\Hom([2], \mathbf{C})}{\Hom([2],
    \mathbf{C})}{\Hom([1], \mathbf{C})} is Cartesian.  But this claim
  is equivalent to the statement that a commutative square in
  $\mathbf{C}$ is the same as two compatible commutative triangles,
  which is obvious.
\end{proof}

\begin{defn}
  The functor $\mathrm{N}$ preserves products, being a right adjoint,
  and so induces a functor $\mathrm{N}_{*} \colon \Cat_{2}
  \to \sCat$, given by applying $\mathrm{N}$ on the mapping spaces;
  this has a left adjoint $\mathrm{C}_{*} \colon \sCat \to
  \Cat_{2}$ given by composition with $\mathrm{C}$,
  since $\mathrm{C}$ preserves products by Proposition~\ref{propn:Cprod}.
\end{defn}

We now briefly recall the definition of the coherent nerve functor
from simplicial categories to simplicial sets, following
\cite{HTT}*{\S 1.1.5}:
\begin{defn}
  Let $P_{i,j}$ be
  the partially ordered set of subsets of $\{i, i+1,\ldots,j\}$
  containing $i$ and $j$. Then $\mathfrak{C}(\Delta^{n})$ denotes the simplicial category with
  objects $0, \ldots, n$ and 
  \[ \mathfrak{C}(\Delta^{n})(i,j) =
  \begin{cases}
    \emptyset, & i > j\\
    \mathrm{N}P_{i,j}, & \txt{otherwise}
  \end{cases}
  \]
  Composition of morphisms is induced by union of the subsets in the $P_{i,j}$'s.
\end{defn}
\begin{remark}
  The simplicial set $\mathrm{N}P_{i,j}$ is isomorphic to
  $(\Delta^{1})^{\times (j - i - 1)}$ for $j > i$.
\end{remark}

\begin{defn}
  The \emph{coherent nerve} is the functor $\mathfrak{N} \colon \sCat
  \to \sSet$ defined by \[\mathfrak{N}\mathbf{C}_{k} =
  \Hom(\mathfrak{C}(\Delta^{k}), \mathbf{C}).\] This has a left
  adjoint $\mathfrak{C} \colon \sSet \to \sCat$, which is the unique
  colimit-preserving functor extending the cosimplicial simplicial
  category $\mathfrak{C}(\Delta^{\bullet})$.
\end{defn}

\begin{defn}
  Let $\mathrm{N}_{2} \colon \Cat_{2} \to \sSet$ denote the
  composite
  \[ \Cat_{2} \xto{\mathrm{N}_{*}} \sCat
  \xto{\mathfrak{N}} \sSet.\]
  This functor has a left adjoint $\mathrm{C}_{2}$, which is the
  composite
  \[ \sSet \xto{\mathfrak{C}} \sCat \xto{\mathrm{C}_{*}} \Cat_{2}.\]
\end{defn}

\begin{remark}
  It is clear from the definitions given in \cite{Duskin}*{\S\S 6.1--6.7} that the
  functor $\mathrm{N}_{2}$ as we have defined it is simply the
  restriction of Duskin's nerve for bicategories to strict
  2-categories. (This nerve also implicitly appeared earlier in
  \cite{StreetAlgOrSx}.)
\end{remark}

\begin{remark}\label{rmk:nervesmall}
  We can describe the strict 2-category $\mathrm{C}_{2}(\Delta^{n})$
  as follows: its objects are $0$, \ldots, $n$. For $i > j$, the
  category $\mathrm{C}_{2}(\Delta^{n})(i, j)$ is empty, and for $j >
  i$ it is the partially ordered set $P_{i,j}$ (which is isomorphic to
  $[1]^{\times (j -i-1)}$ if $j > i$). We can thus describe the
  low-dimensional simplices of the nerve $\mathrm{N}_{2}\mathbf{C}$
  of a strict 2-category $\mathbf{C}$ as follows:
  \begin{itemize}
  \item The 0-simplices are the objects of $\mathbf{C}$.
  \item The 1-simplices are the 1-morphisms of $\mathbf{C}$.
  \item A 2-simplex in $\mathrm{N}_{2}\mathbf{C}$ is given by objects
    $x_{0},x_{1},x_{2}$, 1-morphisms $f_{01} \colon  x_{0}\to
    x_{1}$, $f_{12} \colon x_{1} \to x_{2}$, $f_{02} \colon x_{0} \to
    x_{2}$, and a 2-morphism $\phi_{012} \colon f_{02} \To f_{12} \circ f_{01}$.
  \item A 3-simplex is given by 
    \begin{itemize}
    \item objects $x_{0},x_{1},x_{2},x_{3}$,
    \item 1-morphisms $f_{ij} \colon x_{i} \to x_{j}$ for $0 \leq i <
      j \leq 3$,
    \item 2-morphisms $\phi_{012} \colon f_{02} \To f_{12} \circ
      f_{01}$,
      $\phi_{123} \colon f_{13} \To f_{23} \circ
      f_{12}$,
      $\phi_{023} \colon f_{03} \To f_{23} \circ
      f_{02}$ and $\phi_{013} \colon f_{03} \To f_{13} \circ
      f_{01}$, such that the square
      \csquare{f_{03}}{f_{13}\circ f_{01}}{f_{23}\circ f_{02}}{f_{23}
        \circ f_{12} \circ f_{01}}{\phi_{013}}{\phi_{023}}{\phi_{123}
        \circ \id}{\id \circ \phi_{012}}
      commutes.
    \end{itemize}
  \end{itemize}
\end{remark}

\begin{defn}
  Let $\simp_{\leq k}$ denote the full subcategory of $\simp$ spanned
  by the objects $[n]$ for $n \leq k$. The restriction $\txt{sk}_{k}
  \colon \sSet \to
  \Fun(\simp^{\op}_{\leq k}, \Set)$ has a right adjoint \[\txt{cosk}_{k}
  \colon \Fun(\simp^{\op}_{\leq k}, \Set) \to \sSet.\] We say a
  simplicial set $X$ is \emph{$k$-coskeletal} if it is in the image of
  the functor $\txt{cosk}_{k}$. Equivalently,
  $X$ is $k$-coskeletal if every map $\partial \Delta^{n} \to X$
  extends to a unique $n$-simplex $\Delta^{n} \to X$ when $n > k$.
\end{defn}

\begin{propn}\label{propn:2cat3cosk}
  For every strict 2-category $\mathbf{C}$, the simplicial set
  $\mathrm{N}_{2}\mathbf{C}$ is 3-coskeletal.
\end{propn}

\begin{remark}
  A more general version of this result in the setting of bicategories
  appears in \cite{Duskin}.
\end{remark}

\begin{proof}
  We must show that every map $\partial \Delta^{k} \to
  \mathrm{N}_{2}\mathbf{C}$ extends to a unique map from $\Delta^{k}$
  if $k > 3$. Equivalently, we must show that given a map
  $\mathfrak{C}(\partial \Delta^{k}) \to \mathrm{N}_{*}\mathbf{C}$ it
  has a unique extension to $\mathfrak{C}(\Delta^{k})$ for $k > 3$. We can
  describe the simplicial category $\mathfrak{C}(\partial \Delta^{k})$
  and its map to $\mathfrak{C}(\Delta^{k})$ as follows:
  \begin{itemize}
  \item the objects of $\mathfrak{C}(\partial \Delta^{k})$ are $0,
    \ldots k$,
  \item the maps $\mathfrak{C}(\partial \Delta^{k})(i,j) \to
    \mathfrak{C}( \Delta^{k})(i,j)$ are isomorphisms except when $i =
    0$ and $j = k$,
  \item the simplicial set $\mathfrak{C}(\partial \Delta^{k})(0,k)$ is
    the boundary of the $(k-1)$-cube $\mathfrak{C}(\Delta^{k})(0,k)
    \cong (\Delta^{1})^{\times (k-1)}$.
  \end{itemize}
  Thus extending a map $F \colon \mathfrak{C}(\partial \Delta^{k}) \to
  \mathrm{N}_{*}\mathbf{C}$ to $\mathfrak{C}(\Delta^{k})$ amounts to
  extending the map \[\mathfrak{C}(\partial \Delta^{k})(0,k) \to
  \mathrm{N}\mathbf{C}(F(0), F(k))\] to
  $\mathfrak{C}(\Delta^{k})(0,k)$. But the inclusion
  $\mathfrak{C}(\partial \Delta^{k})(0,k) \to
  \mathfrak{C}(\Delta^{k})(0,k)$ is a composition of pushouts of inner
  horn inclusions and the inclusion $\partial \Delta^{k-1} \to
  \Delta^{k-1}$, and if $k-1 > 2$ the nerve of a category has unique
  extensions along these.
\end{proof}

\begin{thm}[Duskin, Bullejos--Faro--Blanco]
  Suppose $\mathbf{C}$ and $\mathbf{D}$ are strict
  2-categories. Then the maps of simplicial sets
  $\mathrm{N}_{2}\mathbf{C} \to \mathrm{N}_{2}\mathbf{D}$ can be
  identified with the normal oplax functors $\mathbf{C} \to
  \mathbf{D}$.
\end{thm}
\begin{remark}
  The more general version of this result for bicategories appears to
  be an unpublished result of Duskin; for 2-categories it is proved by
  Bullejos, Faro, and Blanco as \cite{BullejosFaroBlanco}*{Proposition
    4.3}.  We do not include a complete proof here, but we will now
  briefly indicate how a map of nerves gives rise to a normal oplax
  functor. By Proposition~\ref{propn:2cat3cosk}, a map
  $\mathrm{N}_{2}\mathbf{C} \to \mathrm{N}_{2}\mathbf{D}$ can be
  identified with a map
  $F \colon \txt{sk}_{3}\mathrm{N}_{2}\mathbf{C} \to
  \txt{sk}_{3}\mathrm{N}_{2}\mathbf{D}$.
  Using Remark~\ref{rmk:nervesmall} we can identify this with the data
  of a normal oplax functor as given in Definition~\ref{defn:oplax}:
  \begin{itemize}
  \item The $0$-simplices of $\mathrm{N}_{2}\mathbf{C}$ are the
    objects of $\mathbf{C}$, so $F$ assigns an object $F(c) \in
    \mathbf{D}$ to every $c \in \mathbf{C}$, which gives (a)
  \item The 1-simplices of $\mathrm{N}_{2}\mathbf{C}$ are the
    1-morphisms in $\mathbf{C}$, with sources and targets given by the
    face maps $[0] \to [1]$, so $F$ assigns a 1-morphism $F(f) \colon
    F(x) \to F(y)$ to every 1-morphism $f \colon x \to y$ in
    $\mathbf{C}$, which gives (b).
  \item Moreover, identity $1$-morphisms correspond to degenerate
    edges in $\mathrm{N}_{2}\mathbf{C}$, so since these are preserved
    by any map of simplicial sets we get $F(\id_{x}) = \id_{F(x)}$,
    i.e.\  (i).
  \item The 2-simplices of $\mathrm{N}_{2}\mathbf{C}$ are given by
    three 1-morphisms $f \colon x \to y$, $g \colon y \to z$, $h
    \colon z \to w$ (corresponding to the three face maps), and a
    2-morphism $\phi \colon h \To g \circ f$. In particular:
    \begin{itemize}
    \item Considering 2-simplices where the second edge is degenerate,
      which correspond to 2-morphisms in $\mathbf{C}$, we see that $F$
      assigns a 2-morphism $F(\phi) \colon F(h) \To F(g)$ to every
      $\phi \colon h \To g$ in $\mathbf{C}$, which gives (c).
    \item Considering 2-simplices where the 2-morphism $\phi$ is the
      identity, we see (as this condition is not preserved by $F$)
      that $F$ assigns a 2-morphism $F(g\circ f) \To F(g) \circ F(f)$
      to all composable pairs of 1-morphisms, which gives (d).
    \end{itemize}
  \item Since $F$ preserves degenerate $2$-simplices, which correspond
    to identity 2-morphisms of the form $f \circ \id \To f$ and $\id
    \circ f \To f$, we get (ii) and (iv).
  \item  The 3-simplices of $\mathrm{N}_{2}\mathbf{C}$ are given by
    \begin{itemize}
    \item objects $x_{0},x_{1},x_{2},x_{3}$,
    \item 1-morphisms $f_{ij} \colon x_{i} \to x_{j}$ for $0 \leq i <
      j \leq 3$,
    \item 2-morphisms $\phi_{012} \colon f_{02} \To f_{12} \circ
      f_{01}$,
      $\phi_{123} \colon f_{13} \To f_{23} \circ
      f_{12}$,
      $\phi_{023} \colon f_{03} \To f_{23} \circ
      f_{02}$ and $\phi_{013} \colon f_{03} \To f_{13} \circ
      f_{01}$, such that the square
      \csquare{f_{03}}{f_{13}\circ f_{01}}{f_{23}\circ f_{02}}{f_{23}
        \circ f_{12} \circ f_{01}}{\phi_{013}}{\phi_{023}}{\phi_{123}
        \circ \id}{\id \circ \phi_{012}}
      commutes.
    \end{itemize}
    In particular, we have:
    \begin{itemize}
    \item If $x_{1} = x_{2} = x_{3}$, $f_{12} = f_{13} = f_{23} =
      \id_{x_{1}}$, and $\phi_{123} = \id_{\id_{x_{1}}}$, then this
      says $\phi_{013} = \phi_{012}\circ \phi_{023}$, and since $F$
      preserves identities this gives (iii).
    \item In the case where the 2-morphisms are all identities, we get (vi).
    \item To get (v), we consider the 3-simplices where $f_{12} =
      \id$, $\phi_{023} = \id$, and $\phi_{013}$ is the composite of
      $\phi_{012}$ and $\phi_{123}$.
    \end{itemize}
  \end{itemize}
\end{remark}

\begin{defn}
  The inclusion $\Gpd \hookrightarrow \Cat$ of the category of
  groupoids preserves products, and so induces a functor
  $\Cat_{(2,1)} \to \Cat_{2}$; we write 
  $\mathrm{N}_{(2,1)}$ for the composite \[\Cat_{(2,1)} \to
  \Cat_{2} \xto{\mathrm{N}_{2}} \sSet.\]
\end{defn}

\begin{cor}\label{cor:pseudosimpl}
  If $\mathbf{C}$ and $\mathbf{D}$ are strict (2,1)-categories, then
  a morphism of simplicial sets $\mathrm{N}_{(2,1)}\mathbf{C} \to
  \mathrm{N}_{(2,1)}\mathbf{D}$ can be identified with a
  normal pseudofunctor $\mathbf{C} \to \mathbf{D}$.
\end{cor}

\begin{defn}
  Recall that a \emph{relative category} is a category $\mathbf{C}$
  equipped with a subcategory $W$ containing all isomorphisms; see
  \cite{BarwickKanRelCat} for a more extensive discussion. A functor
  of relative categories $f \colon (\mathbf{C}, W) \to (\mathbf{C}',
  W')$ is a functor $f \colon \mathbf{C} \to \mathbf{C}'$ that takes
  $W$ into $W'$. We write $\txt{RelCat}_{(2,1)}$ for the strict
  (2,1)-category of relative categories, functors of relative
  categories, and all natural isomorphisms between these.
\end{defn}

We now want to prove that a normal pseudofunctor to
$\txt{RelCat}_{(2,1)}$ determines a map of \icats{} to $\CatI$ via the
following construction:
\begin{defn}
  If $(\mathbf{C}, W)$ is a relative category, let
  $\mathrm{L}(\mathbf{C}, W) \in \sSet^{+}$ be the marked simplicial
  set $(\mathrm{N}\mathbf{C}, \mathrm{N}W_{1})$. This defines a
  simplicial functor $\mathrm{N}_{*}\txt{RelCat}_{(2,1)} \to
  \sSet^{+}$.
\end{defn}

\begin{defn}
  If $(\mathbf{C}, W)$ is a relative category, we write
  $\mathbf{C}[W^{-1}]$ for the \icat{} obtained by taking a fibrant
  replacement of the marked simplicial set $\mathrm{L}(\mathbf{C},
  W)$. More generally, if $\mathbf{C}$ is a strict (2,1)-category and
  $W$ is a collection of 1-morphisms in $\mathbf{C}$, we write
  $\mathbf{C}[W^{-1}]$ for the \icat{} obtained by fibrantly replacing
  the marked simplicial set $(\mathrm{N}_{(2,1)}\mathbf{C}, W)$.
\end{defn}

\begin{lemma}\label{lem:pseudorel}
  Let $\mathbf{C}$ be a strict (2,1)-category, and let $F$ be a normal
  pseudofunctor $F \colon \mathbf{C} \to \txt{RelCat}_{(2,1)}$. If $W$
  is a collection of 1-morphisms in $\mathbf{C}$ such that $F$ takes
  the morphisms in $W$ to weak equivalences of relative categories,
  then $F$ determines a functor of \icats{} $\mathbf{C}[W^{-1}] \to
  \CatI$, which sends $x \in \mathbf{C}$ to
  $\mathbf{E}_{x}[W_{x}^{-1}]$ where $F(x) = (\mathbf{E}_{x}, W_{x})$.
\end{lemma}
\begin{proof}
  By Proposition~\ref{cor:pseudosimpl} the normal pseudofunctor $F$
  corresponds to a map of simplicial sets
  $\mathrm{N}_{(2,1)}\mathbf{C} \to
  \mathrm{N}_{(2,1)}\txt{RelCat}_{(2,1)}$. Composing this with the map
  $\mathfrak{N}(L) \colon \mathrm{N}_{(2,1)}\txt{RelCat}_{(2,1)} \to
  \mathfrak{N}\sSet^{+}$ we get a map $\mathrm{N}_{(2,1)}\mathbf{C}
  \to \mathfrak{N}\sSet^{+}$. We may regard this as a map of marked
  (large) simplicial sets
  \[ (\mathrm{N}_{(2,1)}\mathbf{C}, W) \to (\mathfrak{N}\sSet^{+},
  W'),\] where $W'$ is the collection of marked equivalences in
  $\sSet^{+}$. Now invoking \cite{HA}*{Theorem 1.3.4.20} we conclude
  that $\CatI$ is a fibrant replacement for the marked simplicial set
  $(\mathfrak{N}\sSet^{+}, W')$, so this map corresponds to a map
  $\mathbf{C}[W^{-1}] \to \CatI$ in the \icat{} $\LCatI$ underlying
  the model category of (large) marked simplicial sets.
\end{proof}

We will now make use of Grothendieck's description of pseudofunctors to
the (2,1)-category of categories to get a way of constructing
pseudofunctors to $\txt{RelCat}_{(2,1)}$:
\begin{thm}[Grothendieck~\cite{SGA1}]\label{thm:sga}
  Let $\mathbf{C}$ be a category. Then pseudofunctors from
  $\mathbf{C}^{\op}$ to the strict 2-category $\CAT$ correspond to
  Grothendieck fibrations over $\mathbf{C}$. 
\end{thm}
\begin{remark}
  Let us briefly recall how a pseudofunctor is constructed from a
  Grothendieck fibration, as this is the part of Grothendieck's
  theorem we will actually use. A \emph{cleavage} of a Grothendieck
  fibration $p \colon \mathbf{E} \to \mathbf{B}$ is the choice, for
  each $(e \in \mathbf{E}, f \colon b \to p(e))$, of a single
  Cartesian morphism over $f$ with target $e$; cleavages always exist,
  by the axiom of choice. Given a choice of cleavage of $p$, we define
  the pseudofunctor $\mathbf{C}^{\op} \to \CAT$ by assigning the fibre
  $\mathbf{E}_{b}$ to each $b \in \mathbf{B}$, and for each $f \colon
  b \to b'$ the functor $f^{*}$ assigns to $e \in \mathbf{E}_{b}$ the
  source of the Cartesian morphism over $f$ with target $e$ in the
  cleavage. Clearly, this pseudofunctor will be normal precisely when
  the cleavage is \emph{normal} in the sense that the Cartesian
  morphisms over the identities in $\mathbf{B}$ are all chosen to be
  identities in $\mathbf{E}$. Every Grothendieck fibration obviously
  has a normal cleavage, so from any Grothendieck fibration we can
  construct a normal pseudofunctor.
\end{remark}

\begin{defn}
  A \emph{relative Grothendieck fibration} is a Grothendieck fibration
  $p \colon \mathbf{E} \to \mathbf{C}$ together with a subcategory
  $\mathbf{W}$ of $\mathbf{E}$ containing all the $p$-Cartesian
  morphisms.  In particular, the restricted projection $\mathbf{W}
  \to \mathbf{C}$ is also a Cartesian fibration. Moreover, for every $x \in
  \mathbf{C}$ the fibres $(\mathbf{E}_{x}, \mathbf{W}_{x})$ are
  relative categories, and the functor $f^{*}$ induced by each $f$ in
  $\mathbf{C}$ is a functor of relative categories. If $(\mathbf{C},
  U)$ is a relative category, we say that the relative Grothendieck
  fibration is \emph{compatible with $U$} if this functor $f^{*}
  \colon (\mathbf{E}_{q}, \mathbf{W}_{q}) \to (\mathbf{E}_{p},
  \mathbf{W}_{p})$ is a weak equivalence of relative categories for
  every $f \colon p \to q$ in $U$.
\end{defn}

The following is then an obvious consequence of Theorem~\ref{thm:sga}:
\begin{lemma}\label{lem:relfib}
  Relative Grothendieck fibrations over a category $\mathbf{C}$
  correspond to normal pseudofunctors $\mathbf{C}^{\op} \to
  \txt{RelCat}_{(2,1)}$.
\end{lemma}

\begin{propn}\label{propn:relGrothfib}
  Let $(\mathbf{E}, \mathbf{W})$ be
  a relative Grothendieck fibration over $\mathbf{C}$ compatible with a collection $U$ of
  morphisms in $\mathbf{C}$. Then this induces a functor of
  \icats{} \[ \mathbf{C}[U^{-1}]^{\op} \to \CatI \] that sends $p \in
  \mathbf{C}$ to $\mathbf{E}_{p}[\mathbf{W}_{p}^{-1}]$.
\end{propn}
\begin{proof}
  Combine Lemmas~\ref{lem:relfib} and \ref{lem:pseudorel}.
\end{proof}

All the maps whose naturality we are interested in can easily be
constructed as relative Grothendieck fibrations. We will explicitly
describe this in the case of the unstraightening equivalence, and
leave the other cases to the reader.

\begin{propn}\label{propn:unstrelGr}
  The unstraightening functors $$\txt{Un}_{S}^{+} \colon
  \Fun_{\simp}(\mathfrak{C}(S)^{\op}, \sSet^{+})^{\txt{fib}} \to
  (\sSet^{+})^{\txt{fib}}_{/S}$$ define a relative Grothendieck
  fibration over $\sSet \times
  \Delta^{1}$ compatible with the categorical equivalences in $\sSet$.
\end{propn}
\begin{proof}
  Let $\mathbf{E}$ be the category whose objects are triples $(i, S,
  X)$ where $i = 0$ or $1$, $S \in \sSet$, and $X$ is a fibrant map
  $Y \to S^{\sharp}$ in $\sSet^{+}$ if $i  = 0$ and a fibrant simplicial
  functor $\mathfrak{C}(S)^{\op} \to \sSet^{+}$ if $i = 1$; the morphisms
  $(i, S, X) \to (j, T, Y)$ consist of a morphism $i \to j$ in $[1]$, a
  morphism $f \colon S \to T$ in $\sSet$, and the following data:
  \begin{itemize}
  \item if $i = j = 1$, $X \colon \mathfrak{C}(S) \to \sSet^{+}$ and
    $Y \colon \mathfrak{C}(T) \to \sSet^{+}$, a simplicial natural transformation
    $X \to \mathfrak{C}(f) \circ Y$,
  \item if $i = j = 0$, $X$ is $E \to S$ and $Y$ is $F \to T$, a
    commutative square
    \nolabelcsquare{E}{F}{S^{\sharp}}{T^{\sharp}}
    in $\sSet^{+}$,
  \item if $i = 1$ and $j = 0$, $Y$ is a functor $\mathfrak{C}(S)^{\op} \to
    \sSet^{+}$ and $X$ is $E \to T$, a commutative square
    \nolabelcsquare{\txt{Un}^{+}_{S}(X)}{E}{S}{T.}
  \end{itemize}
  Composition is defined in the obvious way, using the natural maps of
  \cite{HTT}*{Proposition 3.2.1.4}. We claim that the projection
  $\mathbf{E} \to \Delta^{1} \times \sSet$ is a Grothendieck
  fibration. It suffices to check that Cartesian morphisms exist for
  morphisms of the form $(\id_{i}, f)$ and $(0 \to 1, \id_{S})$, which
  is clear.
\end{proof}

\begin{cor}\label{cor:unsteqnat}
  There is a functor of \icats{} $\CatI^{\op} \to \Fun(\Delta^{1}, \LCatI)$
  that sends $\mathcal{C}$ to the unstraightening equivalence
  \[ \Fun(\mathcal{C}^{\op}, \CatI) \isoto \CatICc.\]
\end{cor}


\Addresses
\end{document}